\documentclass[11pt]{article}
\usepackage[a4paper,left=0.85in,right=0.85in,top=0.85in,bottom=1.15in]{geometry}

\usepackage{authblk}
\usepackage[square,numbers]{natbib}
\usepackage{url}

\usepackage{latexsym,amsmath,amssymb,graphics,amscd}
\usepackage{graphicx}
\usepackage{fancyhdr}
\usepackage{enumitem}
\usepackage{amsthm}
\usepackage{cleveref}
\usepackage{caption}
\newtheorem{lemma}{Lemma}[section]

\theoremstyle{definition}

\theoremstyle{remark}
\newtheorem{remark}{Remark}[section]
\newtheorem{example}{Example}[section]

\def\evec{{\bf e}}

\usepackage{marvosym}
\usepackage{enumerate}
\usepackage{enumitem}
\usepackage{times}
\usepackage{framed,color}
\usepackage{mdframed}

\usepackage{authblk}

\def\evec{{\bf e}}

\newcommand{\RR}{\mathbb{R}}
\newcommand{\CC}{\mathbb{C}}

\newcommand{\TT}{\mathbb{T}}
\newcommand{\ZZ}{\mathbb{Z}}

\newcommand{\KK}{\mathbb{K}} 
\newcommand{\cF}{\mathcal{F}}
\newcommand{\cL}{\mathcal{L}}
\newcommand{\Aff}{\mathrm{Aff}} 
\newcommand{\GLin}{\mathrm{GL}} 
\newcommand{\Aut}{\mathrm{Aut}} 

\providecommand{\keywords}[1]
{
  \small	
  \textbf{\textit{Keywords---}} #1
}

\title{Dynamical Systems On Generalised Klein Bottles}

\author{Peter Grindrod\thanks{Mathematical Institute, University of Oxford, Oxford, UK; \Letter: grindrod@maths.ox.ac.uk}  \phantom{ } and Ka Man Yim\thanks{School of Mathematics, Cardiff University, Cardiff, UK;  \Letter:   yimkm@cardiff.ac.uk}}

\begin{document}
\maketitle

\abstract{We propose a high dimensional generalisation of the standard Klein bottle, going beyond those considered previously. 
We address the problem of generating continuous  scalar fields (distributions) and dynamical systems (flows) on such state spaces, which can provide a rich source of examples for future investigations. 
 We consider a class of high dimensional dynamical systems that model distributed  information processing within the human cortex, which may be capable of exhibiting some Klein bottle symmetries. We deploy topological data analytic methods in order to analyse their resulting  dynamical behaviour, and suggesting future challenges. }


\section{Introduction}
Data science typically presents observations (in the form of data) from some implied dynamical system of interest and demands that we character the underlying relevant processes and system behaviour. This includes differentiating one source from another, and other related questions regarding their observed state variables' mutual dependencies and mutual information. 

Here we consider the problem of characterising  finite dimensional manifolds that contain such attractors. The most basic problem is to consider a wide range of the possible finite dimensional manifolds which may be present. This leads us naturally to ask what kind of compact closed manifolds exist and whether we may recognise scalar fields (smooth distributions and potentials) or vector fields (flows) flows defined over them, respecting  their inherent topologies.

For low dimensional state spaces there is a limited set of possible compact closed manifolds. In two dimensions, for example,  we have the 2-sphere, the 2-torus, the Klein bottle and the real projective plane. In higher dimensions these manifolds may generalise to other objects. Generalising spheres and tori is straightforward. 

In this paper we introduce a generalisation of the Klein bottle, going beyond those given previously~\cite{Davis2019AnBottle}. In $k$ dimensions this relies on a partition of independent coordinates into $k_1$ periodic coordinates that may be flipped, or swapped (as if from moving clockwise to moving anticlockwise, for example), and $k_2$ active periodic components which cause the various flips (such that $k=k_1+k_2$). 

By doing so we access a wide range of compact closed manifolds over which we may define, or there may be discovered, dynamical systems. 

We introduce an inverse problem in the form of the  observation of large scale spiking system, including those carrying out information processing within the cortex of the  human brain~\cite{Yamazaki2022SpikingReview},  where the dimension of the attractor may be estimated yet the topological nature of the attractor is elusive. Such problems may become common as high throughput science enables the complete observation of such large scale dynamical systems, which appear to approach unknown attractors of relatively high dimension, that require characterization. Working from observed time series, such as the spike trains as here, requires an analytical pipeline.  The very high state space dimension of such a complex system means that the dimension and topology of the  manifold holding the attractor is essential in gaining an understanding of {\it what is occurring} and {\it how the system responds} to incoming stimuli.

The detailed and rigorous mathematical considerations are deferred to the Appendices.

\section{Generalised Klein Bottles}

For $k_1\ge 1$  and  $k_2\ge 0$ let $B=(B_{i,j})$ denote an irreducible  binary $k_1 \times k_2$ matrix.

Consider $\RR^{k_1+k_2}$  with independent  coordinates,  $x=(x_1, x_2...,x_{k_1})^T$ and    $y=(y_1, y_2...,y_{k_2})^T$,  subject to the following set of  symmetries:
\begin{equation}
(x+a,y+b ) \sim (\varphi(b)x,y)), \ \ (a,b)\in \ZZ^{k_1}\times \ZZ^{k_2}, \label{summy}  
\end{equation}
where $\varphi(b):\ZZ^{k_2} \to {\rm Aut}(\ZZ^{k_1})$ is a $b$-dependent automorphism  defined by the Hadamard (elementwise) product
$$\varphi(b)x=(-1)^{B.b}\odot x = \left((-1)^{(B.b)_1}x_1, \ldots,  (-1)^{(B.b)_{k_1}}x_{k_1} \right)$$
and extended to $\RR^{k_1}$.
So, depending on the parity of the $i$th term of the vector $(-1)^{B.b}$, the sign of the $i$th coordinate, $x_i$, is either flipped (reflected) or else left unchanged.

The symmetries in (\ref{summy}) imply the quotient space is the cube, $[0,1]^{k_1+k_2}$, with no boundaries and endowed locally with the Euclidean metric. We refer to the $x_i$ coordinates  as {\it toroidal coordinates},  since they are 1-periodic. We refer to the $y_j$ coordinates  as {\it Klein coordinates} since  they are 1-periodic and when incremented by 1 they also flip, or reflect, some of the $x_i$ toroidal coordinates (those for which $B_{i,j}=1$). The result is a {\bf generalised Klein bottle}, which we shall denote by $\KK(k_1,k_2,B)$. 

 $B$ specifies  a bipartite graph  indicating which of the toroidal  coordinates are flipped  by which of the Klein coordinates, see Figure \ref{bip}.  The 
 adjacency matrix, $A$, for the associated bipartite graph is given in block matrix form by 
$$
A=\left(
\begin{array}{cc}
0&B\\ 
B^T&0\\ 
\end{array}
\right).
$$
By assumption on $B$ this bipartite graph is strongly connected so that it cannot be decomposed into disjoint connected graphs (since in that case the  space would be a Cartesian product of at least two lower dimensional spaces). Necessarily, every Klein  coordinate  flips at least one toroidal  coordinate, and that every  toroidal coordinate is flipped by at least one Klein  coordinate (for if this were not the case we may discard any  marooned coordinates).  
\begin{figure}[ht]
\centering
\includegraphics[width=0.95\textwidth]{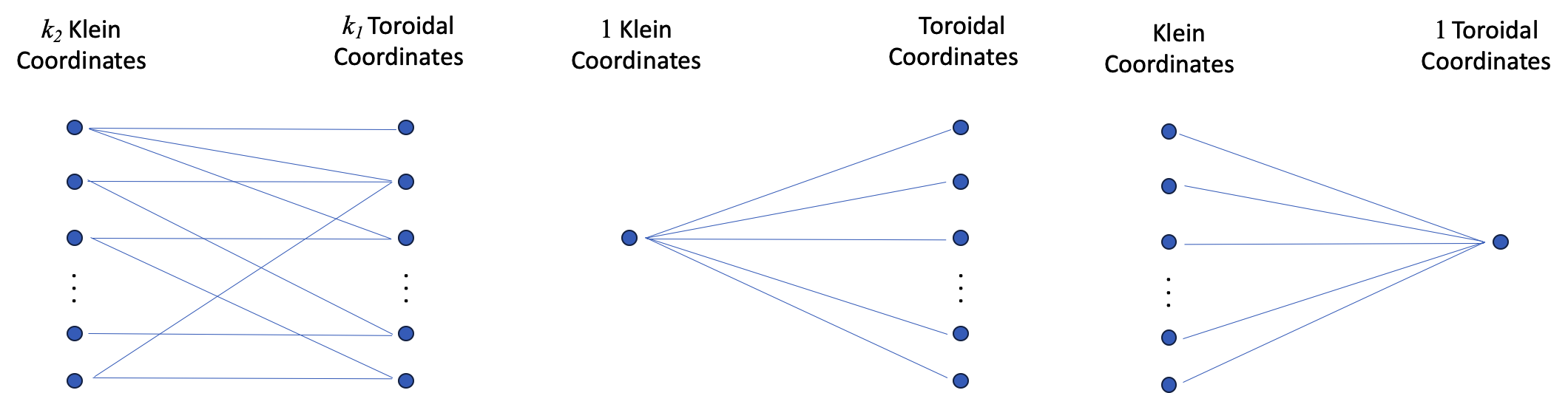} 
\caption{Left: A space $\KK(k_1,k_2,B)$ is specified by a  bipartite graph indicating which of the $k_1$ toroidal coordinates is flipped by which of the $k_2$ Klein coordinates. It must be strongly connected  (irreducible). Centre:  $\KK(k_1,1,B)$ as defined in \cite{Davis2019AnBottle} (with only one possibility for $B$). Right: $\KK(1,k_2,B)$ (again with only one possibility for $B$). }
\label{bip}
\end{figure}

If $k_1=k_2=1$ then we have the standard Klein bottle ($B=(1)$ in that case), a closed non-orientable two-dimensional surface. If $k_2=0$ then $\KK(k_1,0,.)=\TT^{k_1}$ the $k_1$-dimensional torus.
 For spaces $\KK(1,k_2,B)$, for $k_2\ge 2$,  there is but one valid choice of bipartite graph, $B$. Those spaces  are  distinct from the Klein bottle generalisations previously considered in \cite{Davis2019AnBottle} (and its descendants), which are  $\KK(k_1,1,B)$, for some $k_1\ge 2 $, where again there is but one valid choice for $B$, see Figure \ref{bip}.

When two Klein variables only affect the same toroidal variable there is  a {\it hidden} torus. Let $\evec_i$ denote the unit vectors in the direction of the $i$th Klein coordinate, in $\RR^{k_2}$.
 Consider, for example, $\KK(1, n,B)$ for $n\ge 1$. For any distinct pair of the $n$ Klein coordinate symmetries, say those those in $y_i$ and $y_j$ (where without loss of generality  $i < j$), we may deduce
$$
\begin{array}{c}
(x,...y_i+1,...,y_j+1,... ) \sim (x,...y_i,...,y_j ,... ), \\
(x,...y_i-1,...,y_j+1,... ) \sim (x,...y_i,...,y_j ,... ).
\end{array}
$$
Hence there is a 2-dimensional torus, $\TT^2, $ with coordinate directions $(\pm \evec_i +\evec_j)/\sqrt{2}$, and periodicities $ \sqrt{2}$, within the $(y_i,y_j)$-subspace of $\KK(1, n,B)$. 

An inference of the this type  will be true more generally within  $\KK(k_1,k_2,B)$, provided that the subsets of the toroidal coordinates that are  flipped by each of a pair $y_i$ and $y_j$ (and hence the $i$ and $j$th columns of $B$) are identical. 
 Such a situation might occur in scenarios where there is some symmetry within the application, meaning that some of that state variables which are represented by   Klein coordinates are permutable (exchangeable) and thus may cause similar effects (transformations of finite order) on a common subset of the toroidal coordinates. More generally, this  type of hidden symmetry  is  a consequence of the dimesion of the range of $B$, spanned by its columns, is strictly less than $k_2$.

\section{More Symmetries}

For the general case of $\KK(k_1,k_2,B)$, we may rewrite this to consider $B\in \ZZ^{k_1\times k_2 }$,  and  also set  $\varphi : \ZZ^{k_2} \to {\rm Aut}(\ZZ^{k_1})$. Since ${\rm Aut}(\ZZ^{k_1}) = \GLin(k_1, \ZZ)$ is a subgroup of $\GLin(k_1, \RR)$, we define an action on $\RR^{k_1}$, via matrix multiplication:
\begin{equation}
\varphi(b)x=H(b). x \ \ {\rm for}\ b\in \ZZ^{k_2}
\ \ {\rm and\  all}\  x\in \RR^{k_1}, \label{diggy1}    
\end{equation}
 where we introduce  the diagonal matrix (containing diagonal elements in $\{-1,1\}$) given by
\begin{equation}
H(b)\equiv {\rm Diag}\{ (-1)^{B.b} \}  \label{diggy2}   \end{equation}
Then  $H(b).x \equiv (-1)^{B.b} \odot x$ (as before). 

So far  $\varphi$ is subordinate to the choice of $B$. Yet it might be generalised further. Consider a group automorphism $\varphi: \ZZ^{k_2} \to \mathrm{Aut}(\ZZ^{k_1})$ whose image subgroup in $\mathrm{Aut}(\ZZ^{k_1})$ is of finite order. The automorphism group $\mathrm{Aut}(\ZZ^{k_1})$  is the general linear group $\GLin (k_1;\ZZ)$; that is, the group of $k_1\times k_1$ matrices with integer entries, and determinant $\pm 1$. The image of $\varphi$ is an abelian subgroup  of $\GLin (k_1;\ZZ)$, and the homomorphism can be represented by picking $k_2$ matrices $ M_1, \ldots, M_{k_2} \in \GLin (k_1, \ZZ)$ that commute with each other:
\begin{equation}
 \varphi(b) = H(b) \ \   {\rm where}\  H(b)=M_1^{b_1}. \cdots M_{k_2}^{b_{k_2}}.\label{diggy3}    
\end{equation} 
 In the diagonal special case for (\ref{diggy3}),  given by (\ref{diggy1}) and (\ref{diggy2}) and leading to the space $\KK(k_1,k_2,B)$, we have
\begin{equation}
H(b) = D_1^{b_1}. \cdots .D_{k_2}^{b_{k_2}}, \label{diggy4}
\end{equation}
where each $D_i$ is a diagonal matrix with a sequence of $\pm 1$'s down the diagonal. Equation (\ref{diggy4}) represents the  natural factorisation of $H(b)$  as given in (\ref{diggy2}). 

We let  $\KK(k_1,k_2;\varphi)$ denote the space defined by $\RR^{k_1}\times\RR^{k_2}$ modulo the following equivalence relation: for any pair of integers $(a,b)\in \ZZ^{k_1}\times \ZZ^{k_2}$;
\begin{equation} \label{eq:equiv_klein}
    (x+a,y+b) \sim (\varphi(b) x+a, y+b).
\end{equation}
This equivalence relation is in fact the quotient of $\RR^{k_1}\times\RR^{k_2}$ by action of the semi-direct product $G = \ZZ^{k_1} \rtimes_\varphi \ZZ^{k_2}$, on which multiplication is given by 
$$(a',b')\diamond(a,b) \sim (\varphi(b')a +a', b+b').$$ The group $\ZZ^{k_1} \rtimes_\varphi \ZZ^{k_2}$ acts on $\RR^{k_1} \times \RR^{k_2}$ in the following manner:  
\begin{equation}
    (a,b) \cdot(x,y) = ( \varphi(b)x+a, y+b),
\end{equation}
One can easily verify this action is compatible with the group multiplication $\diamond$.

We make some observations about the group action. The action is free:  $(a,b) \cdot (x,y) = (x,y)$ if and only if $(a,b) = (0,0)$, the identity element in $G$. Furthermore, because any non-identity element of $G$ changes of the coordinates of $(x,y)$ by a non-zero, integer amount, the action is a covering space action on $\RR^{k_1} \times \RR^{k_2}$, and the covering of $\RR^{k_1} \times \RR^{k_2}$ of $\KK$ is a regular covering. 

\begin{example}\label{ex:flip}
    We given an example where the image of $\varphi$ in (\ref{diggy3}) is not the subgroup of diagonal matrices. 
For $k_1 = 2$ and $k_2 = 1$; we can have the following homomorphism $\varphi: \ZZ \to \GLin (2,\ZZ)$, where 
\begin{equation}
    \varphi(b)x \equiv H(b).x \  \ {\rm where}\  H(b) = \begin{pmatrix}
0 & 1\\
1 & 0
\end{pmatrix}^b.	
\end{equation}
The action of $(a_1, a_2, b)$ on $(x_1, x_2, y)$ is then given by 
\begin{equation}
\begin{pmatrix}
    x_1 \\
    x_2\\
    y
\end{pmatrix} \mapsto \begin{pmatrix}
    \begin{pmatrix}
0 & 1\\
1 & 0
\end{pmatrix}^b \begin{pmatrix}
x_1 \\ x_2
\end{pmatrix} \\ y 
\end{pmatrix}
    	+ \begin{pmatrix}
a_1 \\ a_2 \\ b
\end{pmatrix} 
\end{equation}
Note that in this case the parity of $b$ controls whether the $x_1$
 and $x_2$ coordinates are swapped (i.e. a reflection across the diagonal of the first two components); whereas previously it controlled the reflection of one or more component coordinates. 
\end{example}

\subsection{Reduction of Symmetries} 
Since the equivalence relations on $ \RR^{k_1} \times \RR^{k_2}$  is due to a quotient by a group action, it suffices to write down the equivalence relations generated by the generators of the group, as they imply all of the other equivalence relations. Since the group $\ZZ^{k_1} \rtimes_\varphi \ZZ^{k_2}$ is generated by the set $\{(e_i, 0), (0, e_j)\}$ where $e_i$'s and $e_j$'s are unit coordinate vectors in $\ZZ^{k_1}$ and $\ZZ^{k_2}$ respectively, we can express the equivalence relations in \cref{eq:equiv_klein} as follows: for all $(x,y) \in \RR^{k_1} \times \RR^{k_2}$,
\begin{equation}
    (x,y)  \sim(x +e_i, y)\ \text{and}\ (x,y)  \sim(\varphi(e_j)x , y+ e_j).
\end{equation}
We can also choose the set of generators to be those that generate the kernel and co-image of $\varphi$. Consider in particular the diagonal case \cref{diggy2}, where
\begin{align*}
   \varphi(b) = {\rm Diag}\{ (-1)^{B.b} \}. 
\end{align*}
Regardless of the choice of $B$, if $b \in (2\ZZ)^{k_2}$, then $\varphi(b) = 1$. We now analyse how the choice of $B$ affects the action. Since $(-1)^{2(B.b)} = 1$ as well, we can trivially write 
\begin{align*}
   \varphi(b) = {\rm Diag}\{ (-1)^{(B.(b\ \mathrm{mod}\ 2))
   \ \mathrm{mod} \ 2} \}. 
\end{align*}
In other words, we regard the binary matrix $B$ as a linear transformation of $\ZZ_2$-vector spaces.
The image of $\varphi$ are $k_1$ dimensional diagonal matrices with $\pm 1$ entries. The set of such matrices forms a subgroup of $\GLin(k_1, \ZZ)$  isomorphic to $\ZZ_2^{k_1}$. The isomorphism is given by
\begin{align}
    \exp: c \in\ZZ_2^{k_1}\mapsto {\rm Diag}\{ (-1)^c \}.
\end{align}
Thus we can express $\varphi$ as the composition that factors through 
\begin{equation}
    \varphi: \ZZ^{k_2} \xrightarrow{\mathrm{mod}\ 2} \ZZ_2^{k_2} \xrightarrow{B} \ZZ_2^{k_1} \xrightarrow[\cong]{\exp}\ZZ^{k_1}_2.
\end{equation}
The kernel $\ker \varphi$ consists of even integers, and  integers (mod 2) in  the kernel of $B$; the image of $\varphi$ is isomorphic to the image of $B$ in $\ZZ_2^{k_1}$. Computationally, the basis vectors of the kernel and image of $B$ can be easily inferred by performing Gaussian elimination on matrices derived from $B$ over the field $\ZZ_2$. Let $r \leq k_2$ be the rank of $B$, and suppose columns $j_1,\ldots, j_r$ of $B$ form a basis for the image of $B$. We can express the equivalence relations as 
\begin{align}
     (x,y)  &\sim(x +e_i, y)   \\
     (x,y)  &\sim(x, y + 2e_j)   \\
     (x,y)  &\sim( x, y+ c) \quad c \in \ker B\\
     (x,y)  &\sim((-1)^{B_{j_k}} \odot x, y+ e_{j_k}) \quad k = 1,\ldots, r. 
\end{align}
Note that $c$ is a vector in $\ZZ_2^{k_2}$, realised as a real vector with $\{0,1\}$ entries in $\RR^{k_2}$. The first two symmetry conditions is  a \emph{toroidal} symmetry on the whole space. The latter two equivalence relations can be expressed as a $\ZZ_2^{k_2}$ action on the torus obtained by imposing the first two equivalence relations on $\RR^{k_1} \times \RR^{k_2}$. We expound on this point of view in \Cref{appA}.

\begin{example}
    Consider the full coupling case, where $B$ is the $k_1 \times k_2$ binary matrix with entries all equal to one. The kernel of $B$ is spanned by $(k_2-1)$ $\ZZ_2$-vectors 
    $$(1,1,0,\ldots, 0),(0,1,1,0,\ldots, 0),\ldots, (0,\ldots, 1,1),$$ 
    and the image of $B$ is simply the vector $(1, \ldots, 1)$. Thus the equivalence relations are generated by 
    \begin{align*}
        (x,y)  &\sim(x +e_i, y)   \\
         (x,y)  &\sim(x, y + 2e_j)   \\
        (x,y)  &\sim( x, y+ e_i + e_{i+1}) \quad i= 1,\ldots, k_2 -1 \\
        (x,y)  &\sim(-x, y+ e_{1}). 
    \end{align*}
\end{example}

\subsection{Functions on Klein Bottles}
We can specify scalar functions on $\KK(k_1, k_2, B)$ uniquely with functions  $F : \RR^{k_1 + k_2} \to \RR$ that satisfy the condition 
\begin{equation}\label{eq:general_symmetry_f}
    F(x,y) = F((a,b) \cdot (x,y)) = F( \varphi(b)x+a, y +b),
\end{equation}
for all $(a,b) \in G = \ZZ^{k_1} \rtimes_\varphi \ZZ_{k_2}$.
Because $\KK(k_1, k_2, B)$ is a quotient of $q: \RR^{k_1 + k_2} \twoheadrightarrow \KK(k_1, k_2)$ any function $f: \KK \to \RR$ admits a lift $F = f \circ q: \RR^{k_1 + k_2} \to \RR$ to a function on the covering space. Alternatively, since $\KK(k_1, k_2)$ is the orbit space of the group action of $\ZZ^{k_1} \rtimes \ZZ^{k_2}$ on $\RR^{k_1 + k_2}$, any function $F$ that is constant on the orbits descends to a function $f$ on $\KK(k_1, k_2)$, such that $F = f \circ q$.

We begin by demonstrating a simple, practical,  method of construction for a wide range of functions, $F$, with the desired properties.

Consider $\KK(k_1,k_2,B)$, as before. Suppose that we for each $i=1,...,k_1$ have  two non-negative smooth functions, $T_{i,1}(y)$ and $T_{i,2}(y)$,  defined  for $y\in \RR^{k_2}$, such that 
$$T_{i,2}(0)=1 \ \ {\rm and} \ \  T_{i,2}(0)=0,$$
and if $B_{i,k}=1$ then $T_{i,1}(y)$ and $T_{i,2}(y)$  exchange their values each time $y_k$ is incremented by 1. Clearly $T_{i,1}(y)$ and $T_{i,2}(y)$ must be 2-periodic in all arguments.  
\label{ANZ}

Then for $b\in \ZZ^{k_2}$, we have
\begin{align}
    (T_{i,1}(y+b), T_{i,2}(y+b))
    = \begin{cases}
    (T_{i,1}(y), T_{i,2}(y))\ \  {\rm if}\  (B.b)_i \ {\rm is\ even,}\\
  (T_{i,2}(y), T_{i,1}(y))\ \  {\rm if}\  (B.b)_i \ {\rm is\ odd.}
\end{cases}\nonumber \end{align}

Consider the following ansatz:
  \begin{equation}
F(x,y)=  
\prod_{i=1}^{k_1} \left( T_{i,1} (y)
f_{i}(x_i) +
  T_{i,2} (y) f_{i}(1-x_i) \right)\  \label{anz}  
\end{equation}
Here the $f_{i}(x_i) $, for $i=1,...,k_1$, are arbitrary   smooth 1-periodic functions; and the pair $T_{i,1}(y)$, and $T_{i,2}(y)$  switch, or interpolate, the corresponding term in the product between $f_i(x_i) $ and its flipped form $f_i(1-x_i) $ as the $y_k$s are incremented. Hence,  being constant along orbits,  $F(x,y)$ descends to a function  defined on 
on $\KK(k_1,k_2,B)$.  It is  separable, and such that, at $y=0$,   we have 
 $F(x,0)=\prod_{j=1}^{k_1} f_{j}(x_j).$

The switching functions, $T_{i,1}(y)$, and $T_{i,2}(y)$, depend on $B$ and may be defined as follows. First consider $\KK(k_1,k_2, B)$, and define the function $S_i(y) $, depending on the $i$th row of $B$: 
$$S_i(y)\equiv \prod_{k=1}^{k_2} 
\left(
\frac{-(1-\cos\pi {y_k})}{2 } 
+
\frac{(1+\cos\pi {y_k})}{2 }  
\right)^{B_{i,k}}, \ \ i=1,...,k_1$$
This product of sums expression for $S_i$ may be  expanded out as  a sum of product terms, with each term containing a product of exactly $k_2$ factors. Then we may define $T_{i,1}(y)$ to be  the sum of all the positive terms within that sum (those terms having an even number of negative factors each of the form $\{-(1-\cos\pi {y_k})\}$ for which  $B_{i,k}=1$), and $-T_{i,2}(y)$ to be  the sum of all the negative terms within that  sum (those terms having an odd number of negative factors each of the form $\{-(1-\cos\pi {y_k})\}$ for which  $B_{i,k}=1$).  
Then, by definition,  $T_{i,1}$, and $T_{i,2}$ are 2-periodic in all of their arguments, and we  have 
$$  S_i(y)=T_{i,1}(y)-T_{i,2}(y),$$ expressing $S_i$ as the difference between two  non-negative continuous functions, where  $S_I(0)=T_{i,1}(0)=1$ and $T_{i,2}(0)=0$. 

Note that we may also write 
$$S_i(y)=\prod_{k=1}^{k_2} 
(\cos \pi y_k)^{B_{i,k}}
.$$
Hence, when  each of the $y_k$'s for which $B_{i,k}=1$ are  independently incremented  by 1 then $S_i$ must change sign, and consequently $T_{i,1}$ and $T_{i,2}$ must exchange their values. The pair thus subsume the Klein symmetries of $\KK(k_1,k_2,B)$.

Defined in this way or otherwise, the pair $T_{i,1}$ and $T_{i,2}$ imply that the function $F$, given by the ansatz in (\ref{anz}), descends to a scalar field on $\KK(k_1,k_2,B)$, as required

In \Cref{appA} we consider the necessary and sufficient conditions for scalar fields satisfying the symmetry condition $F(z) = F(g \cdot z)$, and derive a Fourier basis for such scalar fields. The condition is expressed as $F$ being in the kernel of a linear operator $\mathcal{L}$ on scalar fields, which we describe in \cref{eq:operator_scalar}. This operator can be interpreted as a graph Laplacian; the technical details are briefly summarised in \Cref{rmk:laplacian}.

Using the Fourier transform (which is natural given the periodicities), the symmetry requirements imposed on $F$ imply constraints on the Fourier coefficients of $F$. These constraints are expressed once again linearly: the coefficients are in the kernel of an operator $\mathcal{L}^\star$ on the vector space of Fourier coefficients, which is dual to $\mathcal{L}$ (\cref{eq:dual_operator_scalar}). By taking the inverse Fourier transform on coefficients in the kernel of $\mathcal{L}^\star$, we can find a Fourier basis for scalar functions that satisfy the symmetry constraints.
\begin{example}
     We plot some Fourier modes of scalar fields on the standard Klein bottle in \Cref{fig:fourier_modes_scalar} as an illustration. These are formed by linear combinations of Fourier basis functions of the form
     \begin{align*}
         \cos(2\pi \lambda x) \cos(\pi \zeta y + \phi) \quad &\quad \lambda \in \ZZ_{\geq 0}, \zeta \ \text{even} \\
         \sin(2\pi \lambda x) \cos(\pi \zeta y + \phi) \quad &\quad \lambda \in \ZZ_{\geq 0}, \zeta \ \text{odd}.
     \end{align*}
     These are derived in \Cref{ex:standard_klein} in \Cref{appA}. We can verify that these functions satisfy the symmetry conditions: for integers $a,b$, if $\zeta$ is even:
    \begin{align*}
         \cos(2\pi \lambda ((-1)^b x + a)) \cos(\pi \zeta (y +b) + \phi) &=  \cos(2\pi \lambda x) \cos(\pi \zeta y + \phi).
    \end{align*}
    On the other hand, if $\zeta$ is odd,
    \begin{align*}
         \sin(2\pi \lambda ((-1)^b x + a)) \cos(\pi \zeta (y + b) + \phi) &= (-1)^{b(\zeta + 1)} \sin(2\pi \lambda x)  \cos(\pi \zeta y + \phi) \\
         &=   \sin(2\pi \lambda x) \cos(\pi \zeta y + \phi).
     \end{align*}
     The calculations in \Cref{ex:standard_klein} show that these functions form a basis for any function admitting a Fourier transform while satisfying the symmetry conditions for the Klein bottle. 
\end{example}

\begin{figure}[h]
    \centering
    \includegraphics[width=0.85\linewidth]{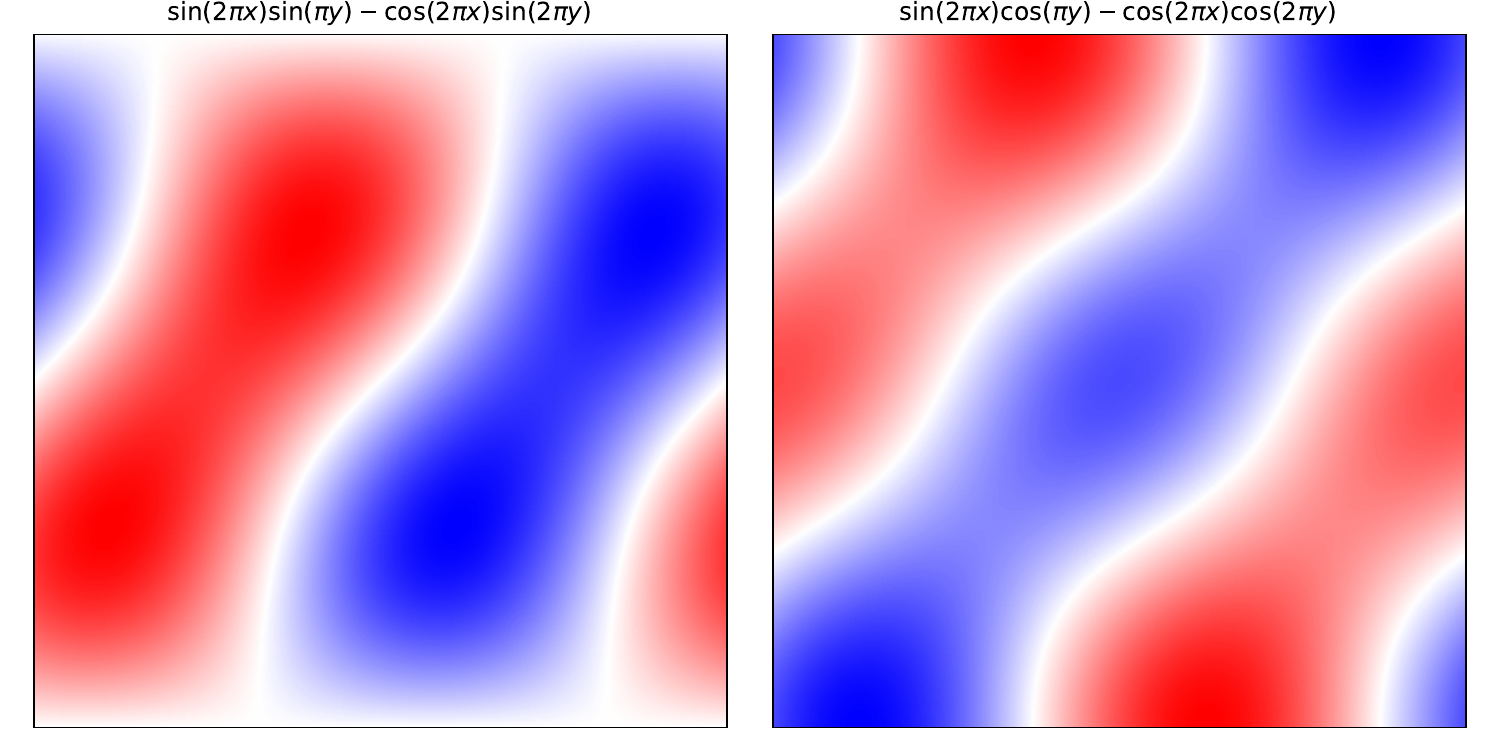}
    \caption{Examples of two scalar fields on the standard Klein Bottle $\KK(1,1)$. The Klein bottle is illustrated here as a unit square domain where the top and bottom sides are identified with opposite orientation, and the left and right hand are identified with the same orientation. The scalar fields are written as linear combinations of Fourier bases functions compatible with the group action on $\RR^2$ that has the Klein bottle as the quotient. The heat maps of these functions are coloured such that white indicates where the function is zero, and the spectrum from red, white to blue goes from positive to negative. }
    \label{fig:fourier_modes_scalar}
\end{figure}

\subsection{Vector Fields on Klein Bottles}
In this section we wish give to give a parameterisation for continuous vector fields on the Klein bottle,  $\KK(k_1,k_2,B)$. Unlike Euclidean space, where vector fields can be expressed by any smooth maps $\RR^{d}$ to $\RR^{d}$,  on a  Klein bottle we may transport a local coordinate frame around one of the non-contractible loops of $\KK(k_1,k_2,B)$ that  flips it. In concrete terms, the non-orientability prevents us  from expressing vectors at different points on $\KK$ with a consistently defined set of global coordinates: we require a flip. In contrast, on the flat torus we may indeed parallel transport a local coordinate frame from one point to any other point, and express vectors at different locations with the a coordinate system that is periodic in each toroidal coordinates.

However, we can lift (pullback) vectors on the Klein bottle to its universal covering space $\RR^{k_1+k_2}$, where we have a global $(x,y)$ coordinate frame. By matching the origin of $\RR^{k_1+k_2}$ and the coordinate system to the local coordinate system of a point in $\KK(k_1,k_2,B)$, the explicit coordinate expression of the vector field lifted in $\RR^{k_1+k_2}$ should reflect the change in the orientation of the coordinate system in $\KK$. We derive the following symmetry condition on lifted vector fields in \Cref{appB}. Writing the lifted vector field $V$ in components $V = (X,Y)$ where $X \in \RR^{k_1}$ and $Y \in \RR^{k_2}$, we have for $(a,b) \in \ZZ^{k_1} \rtimes \ZZ^{k_2}$
\begin{align}
    (-1)^{B.b} \odot  X(x,y) &= X((-1)^{B.b} \odot x + a,y + b) \\
    Y(x,y) &= Y((-1)^{B.b} \odot x + a,y + b) . 
\end{align}
We begin by demonstrating this possibility via a simplifying ansatz (that subsumes the required Klein symmetries above). We provide a more exhaustive and abstract approach in \Cref{appB}.


{ 
First, consider $\KK(k_1,k_2, B)$, and let $T_{i,1}(y)$ and $T_{i,1}(y)$ be defined as before, in Section \ref{ANZ}.

Then consider the following ansatz:
\begin{equation}
X_i(x,y)= \left( T_{i,1} (y)
 f_{i,i}(x_i) -
  T_{i,2} (y) f_{i,i}(1-x_i) \right)  
\prod_{j=1 \  j\ne i}^{k_1} \left( T_{j,1} (y)
 f_{i,j}(x_j) +
  T_{j,2} (y) f_{i,j}(1-x_j) \right)   \label{oma1}  
\end{equation}
  for $i=1,...,k_1;$
  \begin{equation}
Y_i(x,y)=  
\prod_{j=1}^{k_1} \left( T_{j,1} (y)
 g_{i,j}(x_j) +
  T_{j,2} (y) g_{i,j}(1-x_j) \right)\  \label{oma2}  
\end{equation}
for $i=1,...,k_2.$
  
The form in (\ref{oma1}) allows for any of the $x_j$-coordinate dependencies of $X_i$  to be reversed by suitably  incrementing the appropriate $y$-coordinates, and also induces a change of sign in $X_i$ whenever   $x_i$  is reversed.  The form in (\ref{oma2}) allows for the $x_j$-coordinate dependencies of $Y_i$  to be reversed by suitably incrementing the appropriate $y$-coordinates.

The  vector field given in (\ref{oma1}) and (\ref{oma2}) respects the appropriate symmetries for $\KK(k_1,k_2,B)$, which are inherited from fact of  the swapping of the $T_{i,1}$ and $T_{i,2}$ values as appropriate coordinates $y_k$ are increment by 1. Hence it defines a separable vector field over $\KK(k_1,k_2,B)$.

We may of course linearly combine similar separable fields to obtain more general flows on $\KK(k_1,k_2,B)$. This ansatz establishes existence and provides an accessible way to generate fields.

}

In Figures \ref{flo} and \ref{flo3D} we depict the streamlines for two example flows that were  generated in this manner, on $\KK(1,1,(1))$ and $\KK(1,2,(1,1))$ respectively.

\begin{figure}[htbp]
\centering
\includegraphics[width=0.48\linewidth]{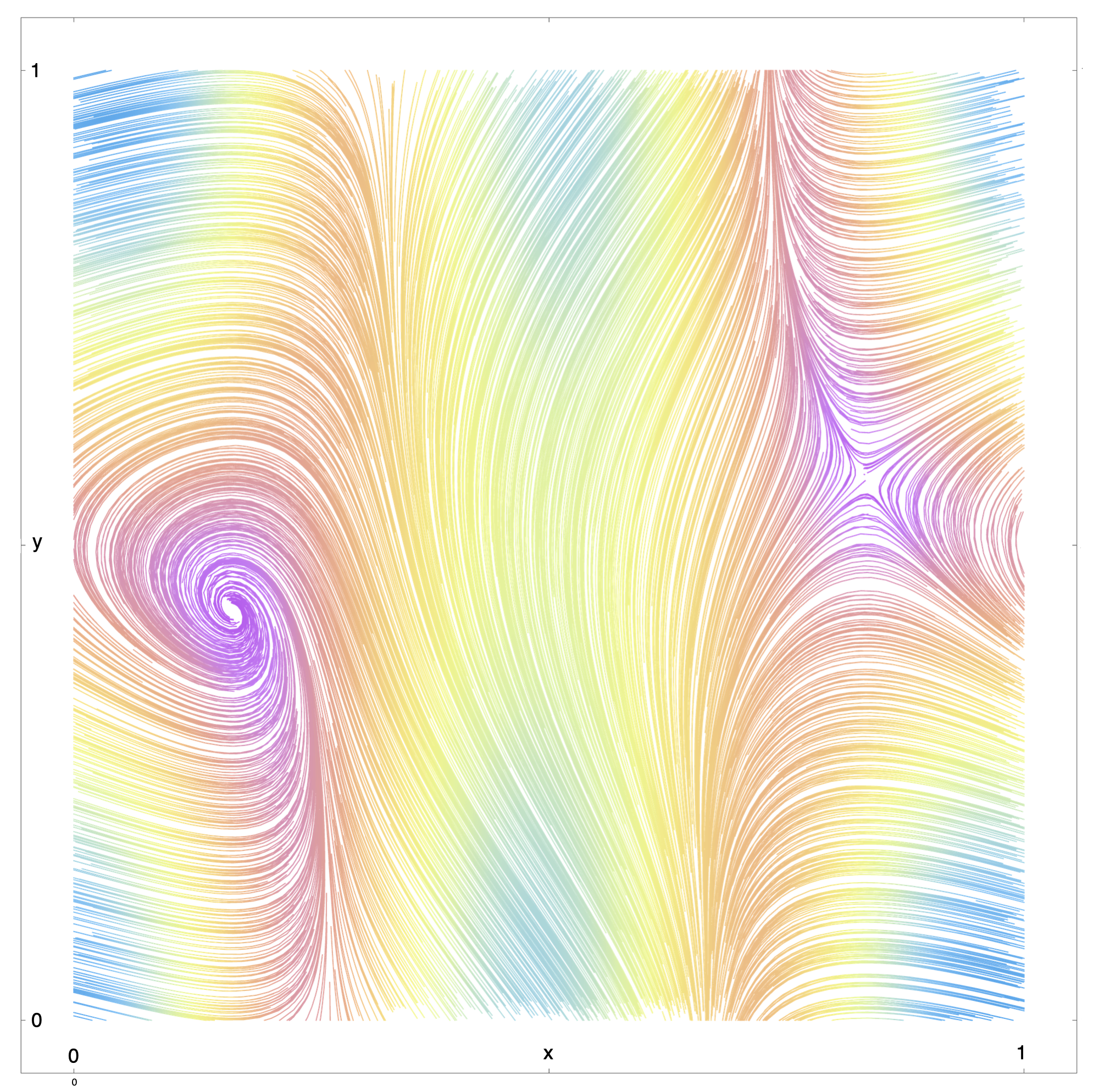}
\includegraphics[width=0.48\linewidth]{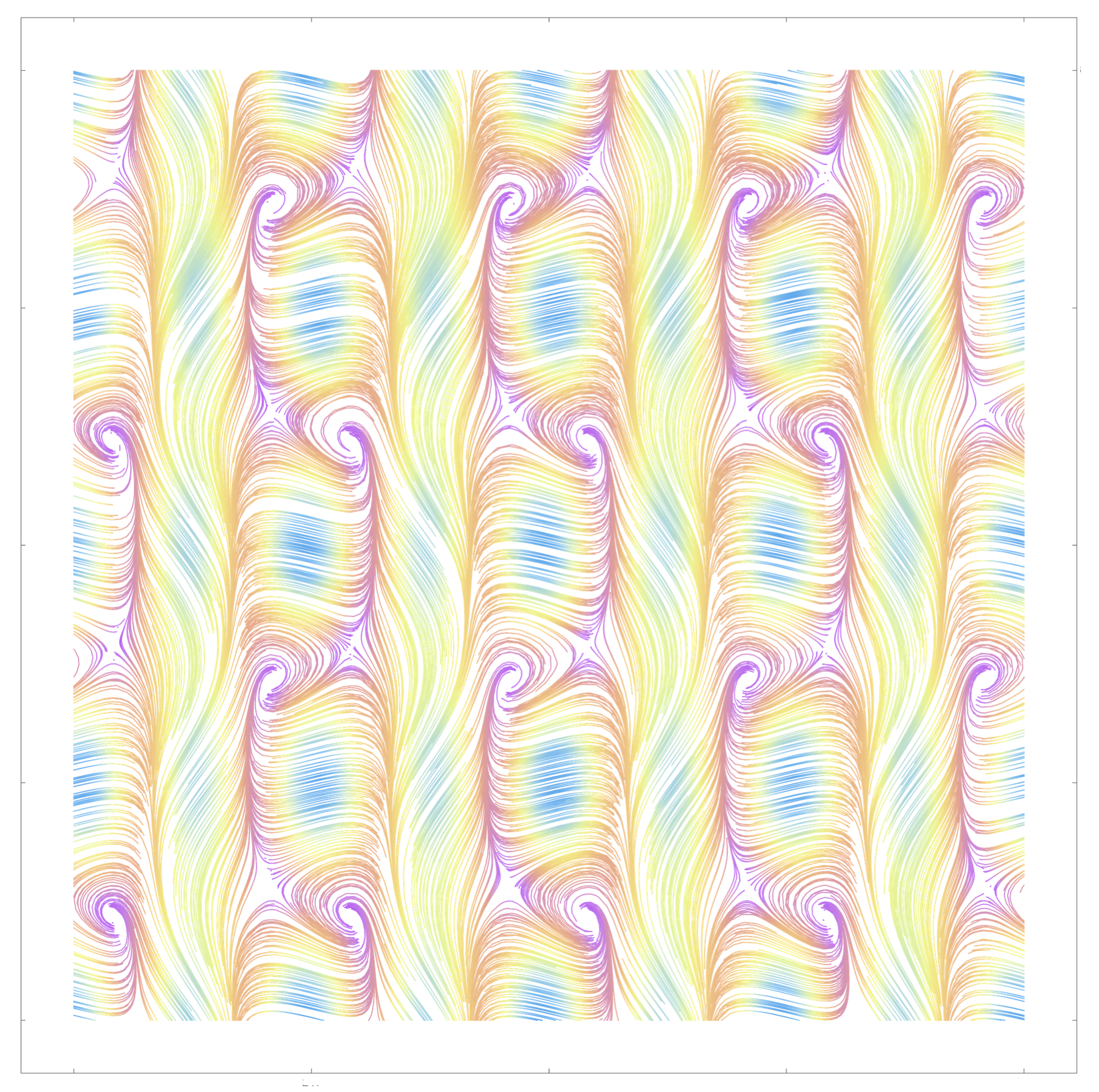}
\caption{Left: Streamlines for an example  flow $(X(x,y),Y(x,y))^T$ on  $\KK(1,1,(1))$, the  standard Klein bottle, with a saddle point and stable focus. Right: the same flow showing the  Klein bottle symmetries on an extended part of the  plane, $[0,4]^2$: the flow is 1-periodic in the toroidal coordinate, $x$, and 2-periodic in the Klein coordinate, $y$, due to the flip symmetry. Streamlines coloured by norm of the vector field (light blue/fast, through yellow then orange then purple/slow).}
\label{flo}
\end{figure}

\begin{figure}[htbp]
\centering

\includegraphics[width=0.48\linewidth]{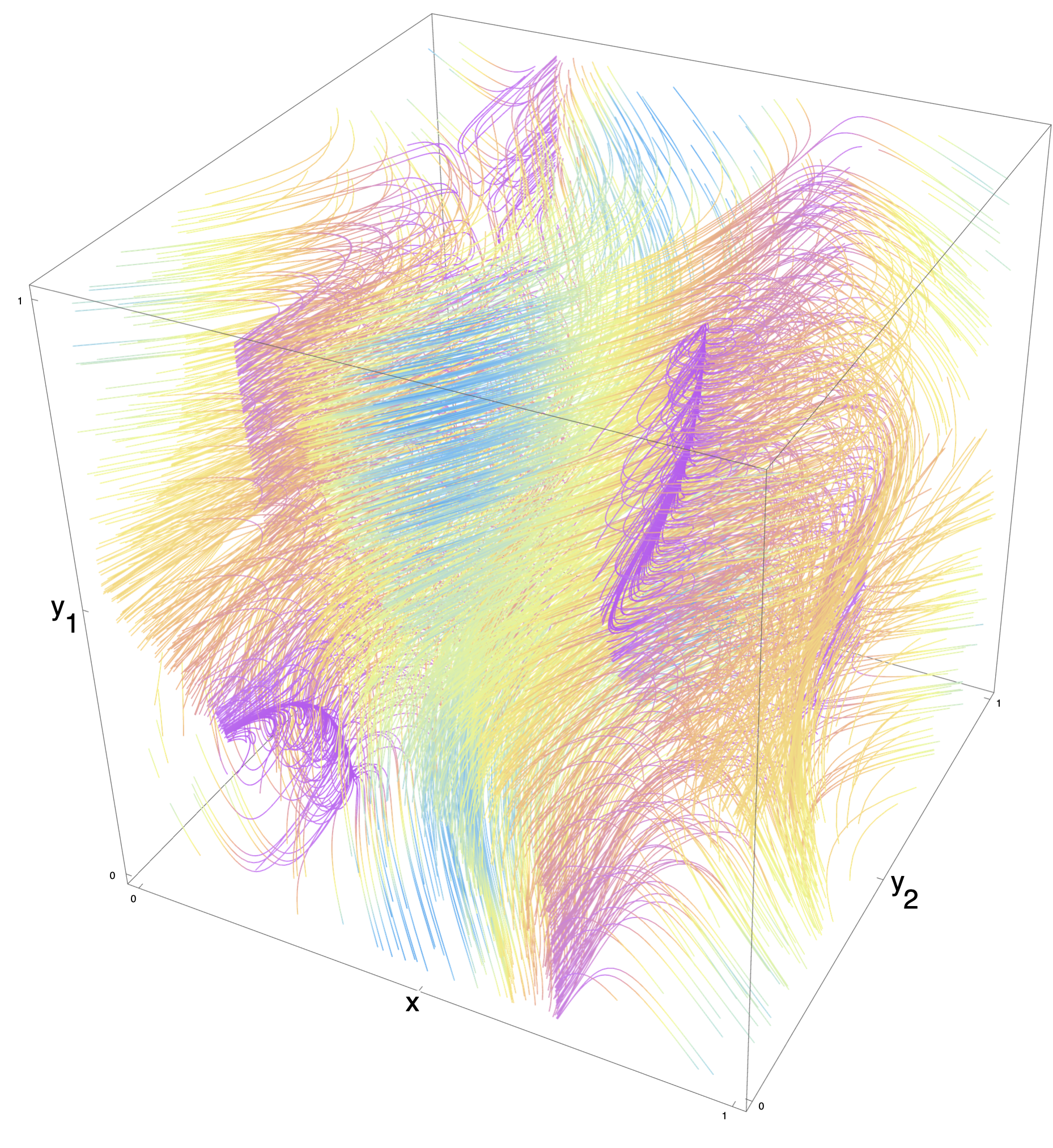}
\includegraphics[width=0.48\linewidth]{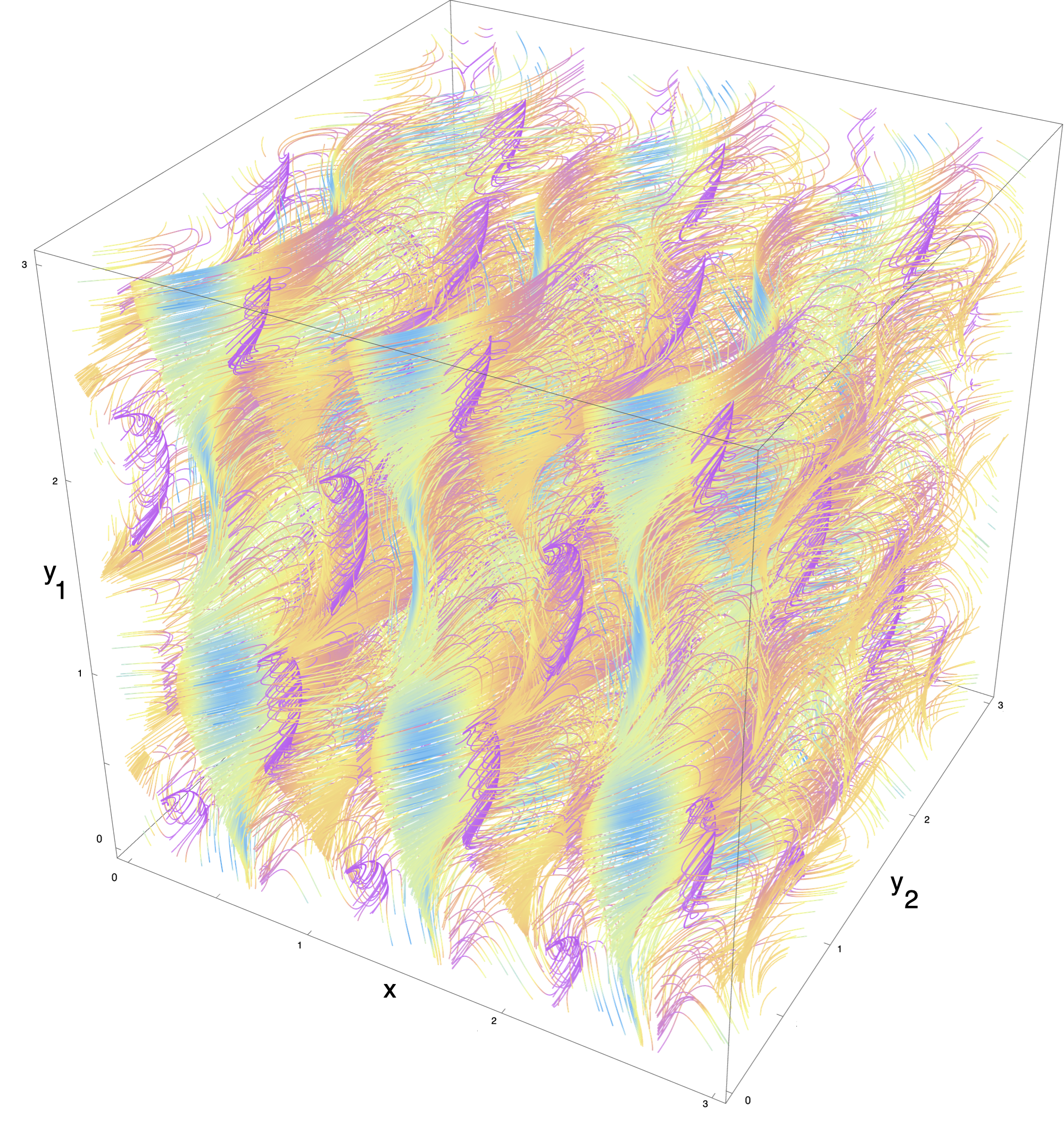}
\caption{Left: Streamlines for an example  flow  $(X(x, y_1,y_2),Y_1(x, y_1,y_2),Y_2(x, y_1,y_2))^T$ on  $\KK(1,2,(1,1))$, with one toroidal and two Klein coordinates, containing a vortex filament. Right: the same flow showing the  $\KK(1,2,(1,1))$ symmetries on an extended volume, $[0,3]^3$: the flow is 1-periodic in the toroidal coordinate, $x$, and 2-period in both Klein coordinates, $y_1$ and $y_2$, due to the flip symmetries. Streamlines coloured by norm of the vector field (light blue/fast, through yellow then orange then purple/slow)}
\label{flo3D}
\end{figure}


In \Cref{appB} we find a Fourier basis for vector fields on $\KK(k_1, k_2, B)$. Similar to the approach in \Cref{appA}, we first formulate the symmetry condition as being in the kernel of an operator on the vector space of vector fields. This correspondingly induces a constraint on the Fourier coefficients of the vector fields, expressed as being in the kernel of the dual of the aforementioned operator. We give an example of a vector field on the standard Klein bottle $\KK(1,1)$ in \Cref{fig:vector_fourier}, constructed on a Fourier basis.

\begin{figure}[h]
    \centering
\includegraphics[width=0.5\linewidth]{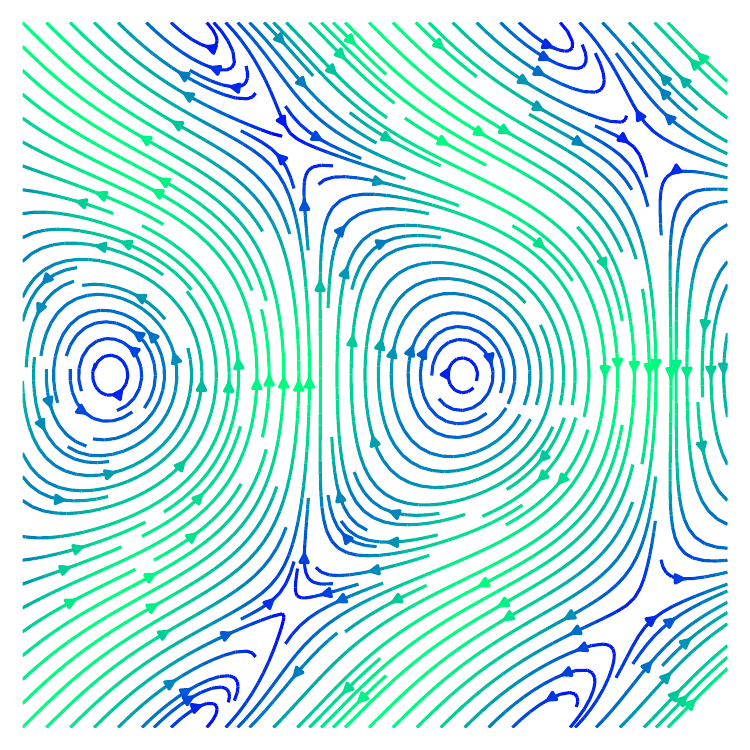}
    \caption{An example of a vector field $(u,v)$ on the Klein bottle, illustrated here as a unit square domain where the top and bottom sides are identified with opposite orientation, and the left and right hand are identified with the same orientation. Here $u(x,y) = \cos(2\pi x) \cos(\pi y) + \sin(2\pi x) \sin(2 \pi y)$, and $v(x,y) = \sin(2\pi x) \sin(\pi y) + \cos(2\pi x) \cos(2 \pi y)$. Note that vectors on the top and bottom edges of the square which are identified with opposite orientation has an extra flip in the $x$-coordinate of the vector field, reflecting the fact that the Klein bottle is not orientable. }
    \label{fig:vector_fourier}
\end{figure}

\section{A challenge: spiking dynamical systems modelling  neural columns}
Here we discuss an application of dynamical systems  with   attractors  set within compact manifolds having topological structures that need to be determined. This raises a number of challenges.

Human brains have evolved both architectures and dynamics to enable effective  information processing with around $10^{10}$ neurons. These {\bf spiking dynamical systems (SDSs)} represent challenges of large-scale state spaces containing  many possible cyclic interactions between neurons. Such brains exhibit  a directed {\it network-of-networks}  architecture, with the {\it inner}, densely connected,  networks of neurons called {\it neural columns} (see \cite{Grindrod2018OnPerspective} and the references therein). Near-neighbouring neural columns have  relatively sparse directed connections  as the {\it outer}  network (between pairs of neurons, one from each). 

Following \cite{Grindrod2017OnCoupling} (and subsequent very large-scale simulations  \cite{Grindrod2021Cortex-LikeWithin}), we consider  the dynamics of information processing  within SDSs representing neural columns, with  $10^2-10^4$ nodes representing  the neurons.   Each node is both {\it excitable} and  {\it refractory}: if the node is in its {\it ready state} when it receives an incoming spike,  from an upstream neighbour,  then it instantaneously fires  and emits an outgoing signal spike along the directed edges to each of its  downstream neighbours, which  takes a small time to  arrive. Once it has fired there must be  a short refractory period whilst the node recovers its local chemical equilibrium, during which time it cannot re-fire and just  ignores any further in-coming spike signals. Once the refractory period is complete the node  returns to the {\it ready state}.  The refractory  time period prevents  arbitrarily fast {\it bursting} (rapid repeat firing)  phenomena.

We will consider sparse directed networks that  are irreducible, so that all nodes may  be influenced by all others.  The whole  network should exhibit  a relatively small diameter. We set an appropriate  set of transit times for each directed edge,   independently and identically drawn from a uniform distribution, and a  common refractory period time, $\delta$,  for all nodes. Then the whole SDS may be kick-started with a single spike at $t=0$ at a particular node, while all other nodes begin in their ready resting state. The  dynamic results in a  firing sequence for each node: a list of firing times at which that node instantaneously spikes. The  SDS begins  to chatter amongst itself  and, after a burn-in period, will   settles down to some very long-term pattern (see \cite{Taube2010InterspikeCells} for an observed instance). It is deterministic and possibly chaotic: see Figures \ref{fig267} and \ref{noper} for  typical examples.

Since the network is irreducible,   we cannot consider sub-networks of the nodes; and  not all independent walks from one specific node to another may be viable: two such walks  may  result in  spike arrival times less that $\delta$ apart; with the  later one  ignored.  
However, the irredicibility does mean that in order to examine the dynamical behaviour of the whole system it is enough to examine  a single node. 
Furthermore, as the spikes are instantaneous, it is conventional within spike sequence analysis to examine the corresponding sequence of successive {\bf inter-spike intervals (ISIs)} \cite{Taube2010InterspikeCells}. By definition  these are reals and are  bounded below by $\delta$. The SDS model timescale is arbitrary: results depend only on the size of $\delta$ relative to the range of the i.i.d. edge transit time.

The SDS examples in Figure \ref{fig267} and Figure \ref{noper} illustrate  these features.  The ISI sequences suggest dynamics within a bounded  attractor, that lies within some manifold, ${\cal M}$ say, of unknown dimension and topology. Generalised Klein bottles and tori  are candidates for ${\cal M}$.

\begin{figure}[htbp]
\centering
\includegraphics[width=0.99\linewidth]{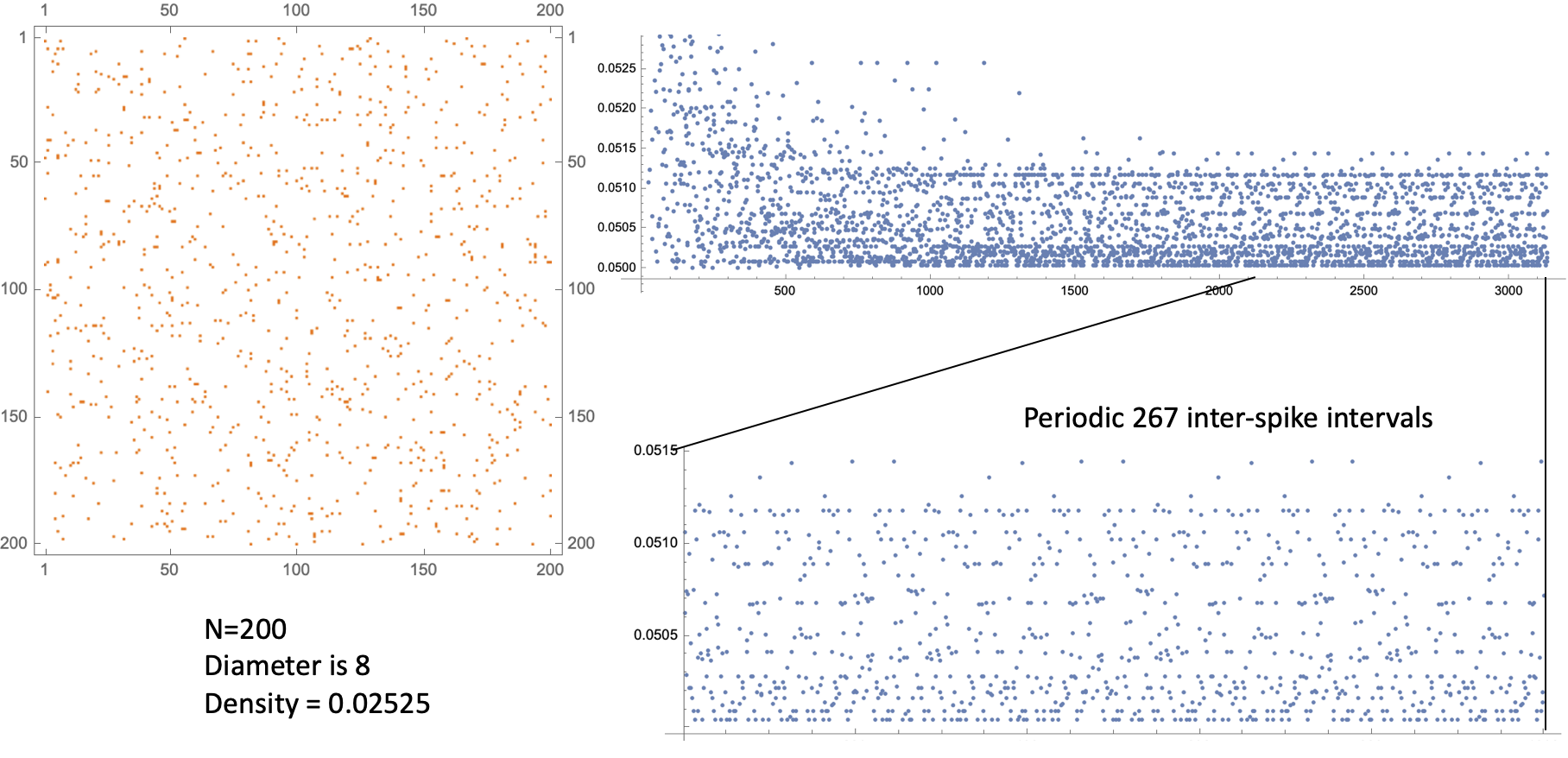}
\caption{Left: An  adjacency matrix for a  directed graph on $N=200$ nodes, with density $\approx$ 2.525\% and  network diameter 8. Right: the successive inter-spike intervals (ISIs) at a single node, after a burn-in period of around 1500 spikes the whole settles down to a long periodic orbit (with 267 successive ISIs) possibly containing many  quasi periodic features due to  cycles within the network. Here the edge transit times are i.i.d. in $U[0.5,1.5]$, and $\delta=0.050$.}
\label{fig267}
\end{figure}

\begin{figure}[htbp]
\centering
\includegraphics[width=0.99\linewidth]{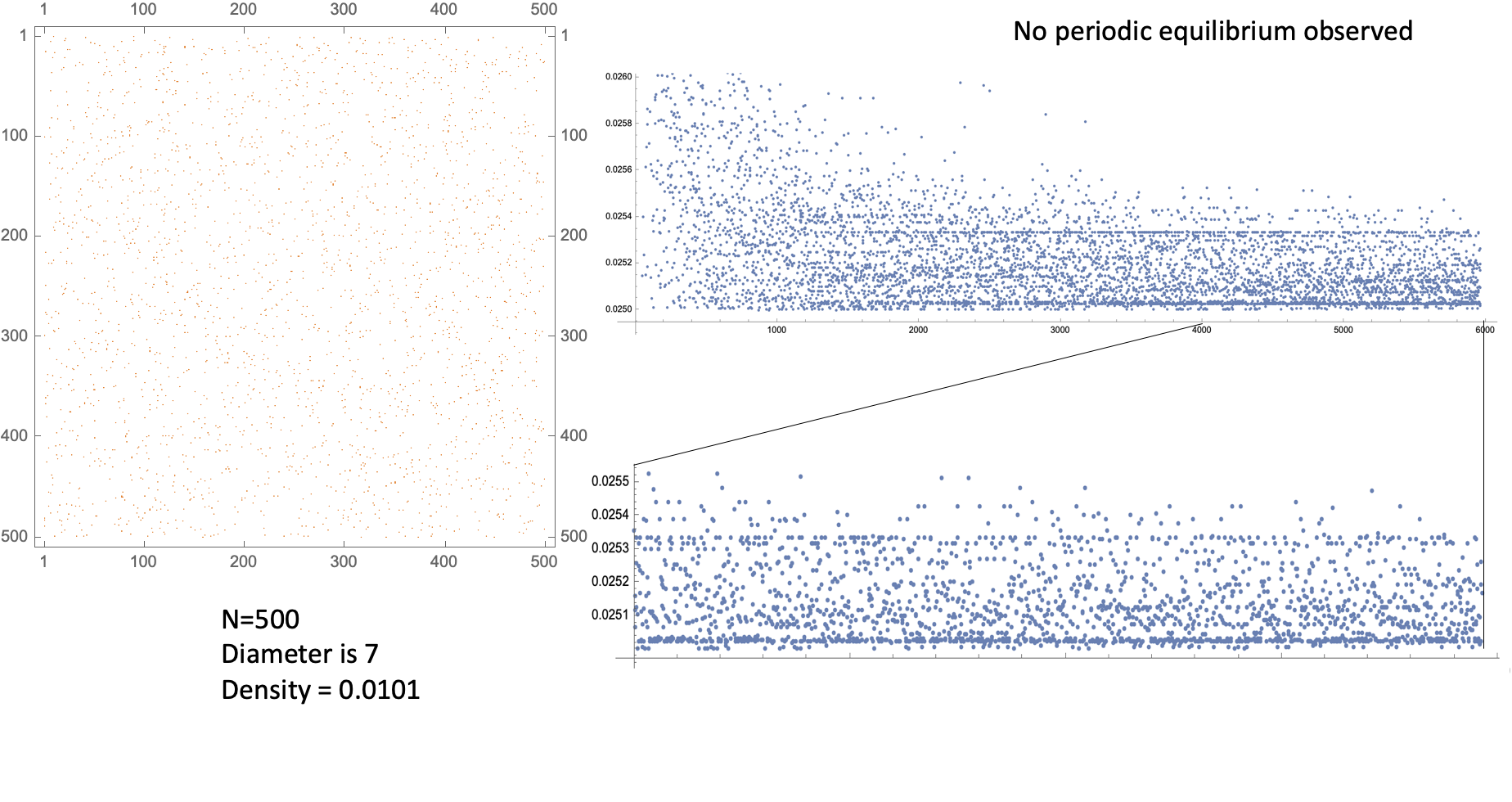}
\caption{Left: An  adjacency matrix for a  directed graph on $N=500$ nodes, with density $\approx$ 1.010\% and  network diameter 7. Right: the successive inter-spike intervals (ISIs) at a single node, after a burn-in period of around 3500 spikes the whole settles down to quasi periodic (possibly chaotic) behaviour, with no long period observed, containing many quasi periodic features due to  cycles within the network. The edge transit times are i.i.d. in U[0.5,1.5], and $\delta=0.025$.}
\label{noper}
\end{figure}


We may embed the observed ISI sequence  as a point cloud within  $\RR^L$ by discarding  the ISI  burn-in, and  moving a window of length $L$ successively along the sequence. Then we may  estimate the dimension, $D_{\cal M}(L)$, of the curved manifold, ${\cal M}\subset \RR^L$, on which that point cloud lives, from the set of all pairwise distances between points within the cloud, via the 2NN (two nearest neighbour) method \cite{Facco2017EstimatingInformation}. For $L$ small the point cloud (and ${\cal M}$) merely fills up the available dimensions, so $D_{\cal M}(L) \sim L$.  As $L$ increases further we will have estimates $D_{\cal M}(L)<<L$ 

Considering  a post burn-in  ISI sequence of length 1900 from the example given  in Figure \ref{noper},   we need to take $L\ge 40$ to avoid any duplicate  points  from windows along the sequence (which we know to be non-periodic).   For lower values of $L$ we may remove any duplicates from the point cloud (which will otherwise interfere with the 2NN algorithm).

We apply this method to  the post burn-in ISI sequence, of length 1900, given  in Figure \ref{noper}  in Figure \ref{zoop} (Left, Green). This  suggests   dimension, $D_{\cal M}$,  of 5 to 7 at the first obvious plateau, embedded in $L=7$ to 11 dimensions.  If $L$ is increased further then  $D_{\cal M}(L)$ increases rather slowly as the attractor {\it fills in} somewhat. For comparison we  show the result for the 267-periodic case, given in Figure \ref{fig267} (Left, Red); where necessarily we have 267 windowed points. This  suggests a  dimension, $D_{\cal M}$,  between 4 and  6.

\begin{figure}[htbp]
\centering
\includegraphics[width=0.48 \linewidth]{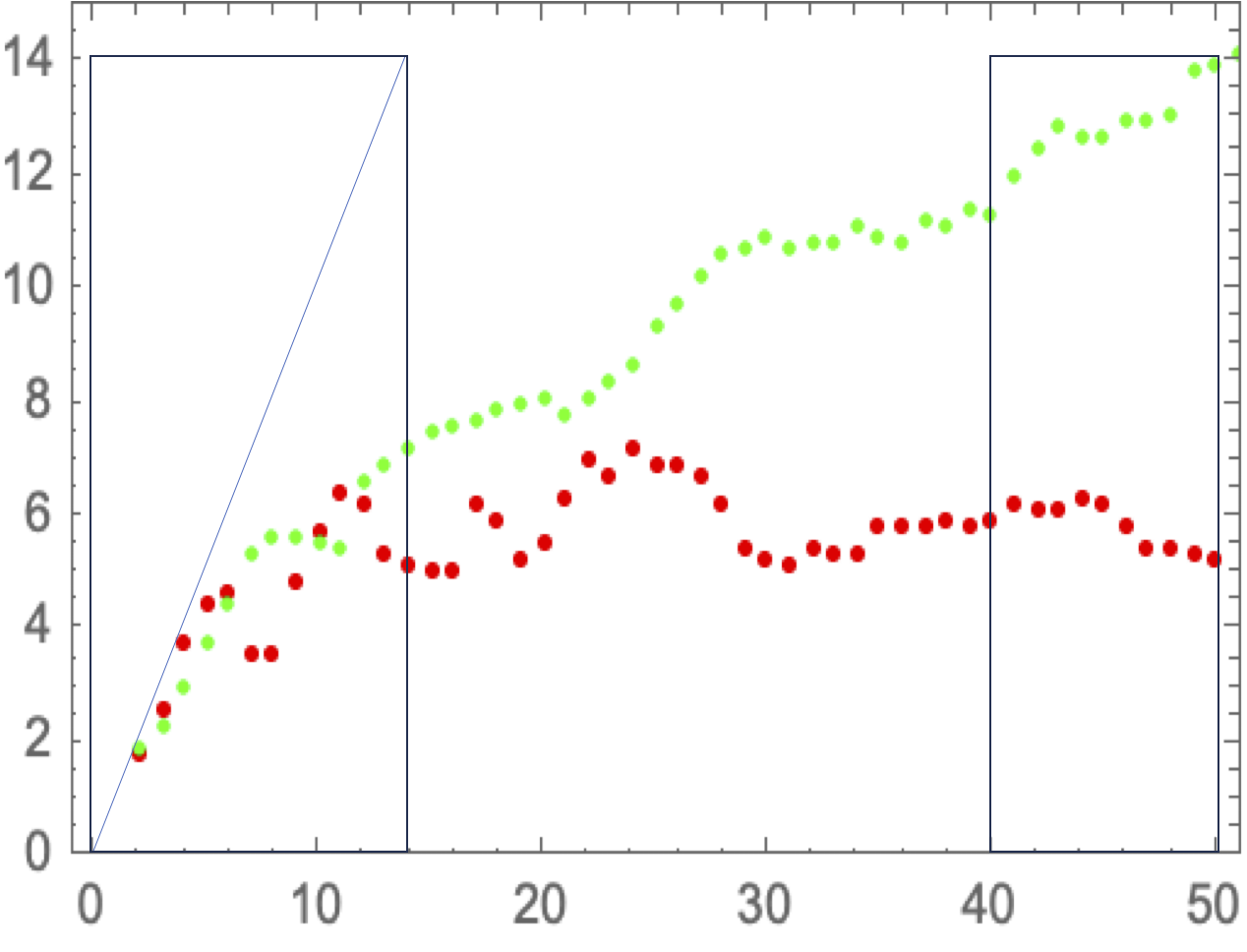}
\includegraphics[width=0.48\linewidth]{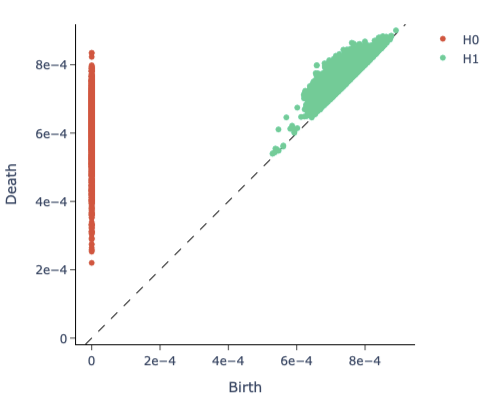}
\caption{Left: Plot of $D_{\cal M}(L)$ versus $L$ using the 2NN method applied to the  moving windows of length $L$, taken from a post burn-in ISI sequence. Green: from a sequence of 1900 points from the non-periodic SDS  given  in Figure \ref{noper}. Red: from a sequence of 267 points from  the 267-periodic SDS  given in Figure \ref{fig267}.
Right: A persistent homology analysis for the point cloud representing the (post burn-in) 267-periodic SDS given in Figure \ref{fig267},  for $L=16$, using  the  Rips filtration (based on pairwise distances between the points). Red points show the ($H_0$) features of connected components as the filtration scale increases, while green points show the ($H_1$) features corresponding to distinct rings around holes through the cloud: such features close to the diagonal are non-persistent and represent sampling noise. }
\label{zoop}
\end{figure}

It is clear that this class of SDSs give rise to highish dimensional compact attractors. What is required is a pipeline that begins from a suitable  embedding of the post burn-in ISIs, into some $\RR^L$ (just as above which determined the dimension, $D_{\cal M}(L)$,  of the manifold, $\cal{M}$, containing the attractor); followed by a computational  method (most likely based on persistent homology) that can determine (a)  the main topological features of the manifold containing the point cloud,  and subsequently (much more challenging)  (b) the topological symmetries that can differentiate between various  generalised Klein bottles of given  dimensions ($k_1+ k_2\approx D_{\cal M}(L)$). 

Following Figure \ref{zoop} (Left), and embedding the 267-periodic ISI from Figure \ref{fig267} into $L=16$ dimensions (whence the  estimate  $D_{\cal M}(L)\approx5.1$),   a persistent homology (PH) analysis is shown in Figure \ref{zoop} (Right). It is somewhat inconclusive,  possibly due to the variations  in the localised cloud density. This issue is the subject of active research. 

It is this characterisation of the SDS's unknown  attracting  manifolds,  that has driven the permissive generalisation of Klein bottles presented in this paper.   This is the subject of active ongoing research. 

\section{Discussion}
In this paper we have introduced a  generalisation of the standard Klein bottle, to higher dimensions, beyond those previously considered. We have focussed on a subclass where the automorphisms are {\it diagonal}, inducing flips (reflections)  of some independent coordinates. 

For all such Klein bottles we have produced both a simple (practical) method, and a rigorous argument, for the definition  of scalar fields (such as smooth probability distributions and potentials) as well as vector fields (flows) that are well-defined.  These constructions are useful when we wish to consider winding flows over Klein bottles, and possibly couple them together.

This class of Klein bottles can be extended to include automorphisms that induce swaps (permutations) between toroidal coordinates (diagonal reflections, rather than individual coordinate reflections). In future work we will extend this approach to consider an even wider class of compact manifolds without boundaries, including real projective geometries, for which there is not a partition of coordinates into Klein variables (controlling the automorphisms),  and toroidal variables (acted on by the automorphisms). Since spheres are the universal covering spaces of projective spaces, we would consider lifts of vector fields on the projective plane to its universal covering sphere to parametrise them, and express their Fourier basis with spherical harmonics that respect the required symmetries. 

Finally, we have set out a challenging application, where high dimensional spiking dynamical systems result in attractors over unknown compact manifolds which  require characterisation. This is highly problematical when the dimension of such spiking systems is very large and the resulting ISI sequences have to be embedded within a suitable space. The identification of the dimension and the topological  properties of resultant attractor manifolds is both a theoretical and computational challenge. 

\subsection*{Acknowledgements}
We are pleased to acknowledge the advice received from Clive E. Bowman and Des J. Higham. KMY would like to thank John Harvey for his mentorship.
PG was funded by EPSRC grant number EP/Y007484/1. KMY is supported by a UKRI Future Leaders Fellowship [grant number MR/W01176X/1; PI: J. Harvey].

The authors declare no conflicts of interest. For the purpose of open access, the authors have applied a CC BY public copyright licence to any Author Accepted Manuscript version arising from this submission.

\bibliography{refs}

\begin{thebibliography}{9}
\providecommand{\natexlab}[1]{#1}
\providecommand{\url}[1]{\texttt{#1}}
\expandafter\ifx\csname urlstyle\endcsname\relax
  \providecommand{\doi}[1]{doi: #1}\else
  \providecommand{\doi}{doi: \begingroup \urlstyle{rm}\Url}\fi

\bibitem[Brown(1987)]{Brown1987FromSurvey}
R.~Brown.
\newblock {From groups to groupoids: a brief survey}.
\newblock \emph{Bull. London Math. Soc}, 19\penalty0 (2):\penalty0 113--134, 1987.

\bibitem[Davis(2019)]{Davis2019AnBottle}
D.~M. Davis.
\newblock {An n-dimensional Klein bottle}.
\newblock \emph{Proceedings of the Royal Society of Edinburgh: Section A Mathematics}, 149\penalty0 (5):\penalty0 1207--1221, 10 2019.
\newblock ISSN 0308-2105.
\newblock \doi{10.1017/prm.2018.73}.

\bibitem[Facco et~al.(2017)Facco, d’Errico, Rodriguez, and Laio]{Facco2017EstimatingInformation}
E.~Facco, M.~d’Errico, A.~Rodriguez, and A.~Laio.
\newblock {Estimating the intrinsic dimension of datasets by a minimal neighborhood information}.
\newblock \emph{Scientific Reports}, 7\penalty0 (1):\penalty0 12140, 9 2017.
\newblock ISSN 2045-2322.
\newblock \doi{10.1038/s41598-017-11873-y}.

\bibitem[Grindrod(2018)]{Grindrod2018OnPerspective}
P.~Grindrod.
\newblock {On human consciousness: A mathematical perspective}.
\newblock \emph{Network Neuroscience}, 2\penalty0 (1):\penalty0 23--40, 3 2018.
\newblock ISSN 2472-1751.
\newblock \doi{10.1162/NETN{\_}a{\_}00030}.
\newblock URL \url{https://direct.mit.edu/netn/article/2/1/23-40/5409}.

\bibitem[Grindrod and Lee(2017)]{Grindrod2017OnCoupling}
P.~Grindrod and T.~E. Lee.
\newblock {On strongly connected networks with excitable-refractory dynamics and delayed coupling}.
\newblock \emph{Royal Society Open Science}, 4\penalty0 (4):\penalty0 160912, 4 2017.
\newblock ISSN 2054-5703.
\newblock \doi{10.1098/rsos.160912}.

\bibitem[Grindrod and Lester(2021)]{Grindrod2021Cortex-LikeWithin}
P.~Grindrod and C.~Lester.
\newblock {Cortex-Like Complex Systems: What Occurs Within?}
\newblock \emph{Frontiers in Applied Mathematics and Statistics}, 7, 9 2021.
\newblock ISSN 2297-4687.
\newblock \doi{10.3389/fams.2021.627236}.

\bibitem[Hansen and Ghrist(2019)]{Hansen2019TowardSheaves}
J.~Hansen and R.~Ghrist.
\newblock {Toward a spectral theory of cellular sheaves}.
\newblock \emph{Journal of Applied and Computational Topology}, 3\penalty0 (4):\penalty0 315--358, 12 2019.
\newblock ISSN 2367-1726.
\newblock \doi{10.1007/s41468-019-00038-7}.

\bibitem[Taube(2010)]{Taube2010InterspikeCells}
J.~S. Taube.
\newblock {Interspike Interval Analyses Reveal Irregular Firing Patterns at Short, But Not Long, Intervals in Rat Head Direction Cells}.
\newblock \emph{Journal of Neurophysiology}, 104\penalty0 (3):\penalty0 1635--1648, 9 2010.
\newblock ISSN 0022-3077.
\newblock \doi{10.1152/jn.00649.2009}.

\bibitem[Yamazaki et~al.(2022)Yamazaki, Vo-Ho, Bulsara, and Le]{Yamazaki2022SpikingReview}
K.~Yamazaki, V.-K. Vo-Ho, D.~Bulsara, and N.~Le.
\newblock {Spiking Neural Networks and Their Applications: A Review}.
\newblock \emph{Brain Sciences}, 12\penalty0 (7):\penalty0 863, 6 2022.
\newblock ISSN 2076-3425.
\newblock \doi{10.3390/brainsci12070863}.

\end{thebibliography}

\newpage
\appendix
\setcounter{section}{0} 

\section{Fourier modes of scalar fields on Klein Bottles} \label{appA} 
Recall the Klein group action on $\RR^{k_1} \times \RR^{k_2}$ described above. Given a homomorphism $\varphi: \ZZ^{k_2} \to \Aut(\ZZ^{k_1})$, we have the action of the group $G = \ZZ^{k_1} \rtimes_\varphi \ZZ^{k_2}$ on $\RR^{k_1} \times \RR^{k_2}$, where for $\lambda \in \ZZ^{k_1}$ and $\zeta \in \ZZ^{k_2}$,
\begin{align}
    (\lambda, \zeta) \cdot (x,y) = (\varphi(\zeta).x + \lambda, y  + \zeta).
\end{align}
Note that action of $G$ on $\RR^n$ is a covering space action  --- for any element not equal to the identity $(0,0)$, the action shifts at least one coordinate by an integer, so there is a sufficiently small neighbourhood $U$ of any point, such that $g \cdot U \cap U = \empty$ for $g \neq (0,0)$.

If $\varphi(\zeta)^2 = 1$ for all $\zeta \in \ZZ^{k_2}$,  then $H = \ZZ^{k_1} \times (2\ZZ)^{k_2}$ is a normal subgroup of $G = \ZZ^{k_1} \rtimes_\varphi \ZZ^{k_2}$. Thus, we can factor the quotient of $\RR^n$ by $G$ into two successive quotients:
\begin{align}
    \RR^n \twoheadrightarrow \RR^n / H  \twoheadrightarrow \RR^n / G. 
\end{align}
The first quotient is the quotient of $\RR^n$ by the action of the subgroup $H = \ZZ^{k_1} \times (2\ZZ)^{k_2}$ on $\RR^n$, where due to $\varphi(2\zeta) = \varphi(\zeta)^2 = 1$,
\begin{equation}
    (\lambda, 2\zeta) \cdot (x,y) = (x + \lambda , y + 2\zeta).
\end{equation}
Thus $\RR^n /H$ is diffeomorphic to the $n$-torus $\TT$ where the aspect ratio of any of the first $k_1$ coordinates to the last $k_2$ coordinates is 1:2. Because the action of $G$ on $\RR^n$ is a covering space action, and $H$ is a normal subgroup, the quotient  $q: \TT = \RR^n / H  \to \RR^n / G$ is a covering obtained by the quotient of $\TT$ by a covering space action of the group $G/H \cong \ZZ_2^{k_2}$ on $\TT$. Writing elements of $\TT$ as $z = (e^{2\pi i x_1}, \ldots, e^{2\pi i x_{k_1}},e^{\pi i y_1}, \ldots, e^{\pi i y_{k_2}})$, the action of element $\beta = (\beta_1, \ldots, \beta_{k_2}) \in \ZZ_2^{k_2}$ on such an element is given by 
\begin{align}
    & \beta \cdot (e^{2\pi i x_1}, \ldots, e^{2\pi i x_{k_1}},e^{\pi i y_1}, \ldots, e^{\pi i y_{k_2}}) \\ \nonumber
    &= (e^{2\pi i (\varphi(\beta)x)_1}, \ldots, e^{2\pi i (\varphi(\beta)x)_{k_1}},e^{\pi i (y_1 + \beta_1) }, \ldots, e^{\pi i (y_{k_2} + \beta_{k_2})}).
\end{align}
If we consider functions $f: K \to \CC$, they lift to a function on the torus $q^\ast f: \TT \to \CC$, defined by the function $q^\ast f = f \circ q$ sending elements of $\TT$ to $K$ by $q$ and evaluating $f$ on the image of those points. The lifted function $q^\ast f$ is called the \emph{pullback} of $f$. Since $q$ is surjective, $q^\ast$ is an injection of $\CC$-valued functions on $K$ into those of $\TT$. A function $h: \TT \to \CC$ is the lift of some function on $K$, iff for every $y \in K$, the function $h$ is constant on $f^{-1}(y)$, that is, the orbit of a representative $z \in f^{-1}(y)$
\begin{equation}
    h(z) = h(g \cdot z) \quad \forall z \in T,\ g \in G.
\end{equation}
We invoke \Cref{lem:symmetry}, which implies $h$ satisfies the symmetry above iff it is in the kernel of a \emph{linear} operator $\cL: \CC[\TT] \to \CC[\TT]$ on the vector space of $\CC[\TT]$ of complex valued functions on $\TT$, where 
\begin{equation}\label{eq:operator_scalar}
    \left(\cL h\right)(z) = h(z) - \frac{1}{2^{k_2}} \sum_{\beta \in \ZZ_2} h(\beta \cdot z).
\end{equation}
If $h$ admits a Fourier transform, then $h$ satisfies the symmetry iff the Fourier coefficients $c \in \CC[\ZZ^{k_1} \times \ZZ^{k_2}]$ of $h$ are in the kernel of the linear operator 
\begin{equation} \label{eq:dual_operator_scalar}
   (\cL^\star c){(\lambda, \zeta)} =  c(\lambda, \zeta) - \frac{1}{2^{k_2}}  \sum_{\beta \in \ZZ_2^{k_2}}(-1)^{\zeta.\beta} c( \varphi(\beta)^T\lambda, \zeta).
\end{equation}
This is a consequence of \Cref{lem:symmetry_fourier}. We further remark that while $\cL^\star$ is an operator on an infinite dimensional space, it can be expressed as a direct sum of countably many finite dimensional operators (\Cref{lem:block}).

We give some examples below.
\begin{example}[The Standard Klein Bottle] The standard Klein bottle has $k_1 = k_2 = 1$. The induced homomorphism $\varphi: \ZZ_2 \to \Aut(\ZZ)$ is given by $\varphi(\beta) = (-1)^\beta$. The dual operator $\cL^\star: \CC[\ZZ \times \ZZ] \to \CC[\ZZ \times \ZZ]$ \cref{eq:dual_operator_scalar} is given by 
\begin{align}
    \cL^\star c{(\lambda, \zeta)} &= c(\lambda, \zeta) - \frac{1}{2}  \sum_{\beta \in \ZZ_2} (-1)^{\zeta.\beta} c( \varphi(\beta)\lambda, \zeta) \\ \nonumber
                                          &= \frac{1}{2} \left(c(\lambda, \zeta) -  (-1)^{\zeta} c( -\lambda, \zeta) \right)
\end{align}
The kernel is given by
\begin{align}
    c \in \ker \cL^\star \iff \begin{cases}
        c(\lambda, \zeta) =  c(-\lambda, \zeta) & \zeta \in 2\ZZ \\ 
         c(\lambda, \zeta) =  -c(-\lambda, \zeta) & \zeta \in 2\ZZ +1.
    \end{cases}
\end{align}
In other words, $\ker \cL^\star$ admits a basis of indicator functions on $\ZZ \times \ZZ$
\begin{align}
    \ker \cL^\star = \mathrm{span}\left( \begin{cases}
        1_{(\lambda, \zeta)} + 1_{(-\lambda, \zeta)}  & \lambda \in \ZZ_{\geq 0},\ \zeta \in 2\ZZ \\ 
         1_{(\lambda, \zeta)} - 1_{(-\lambda, \zeta)}  & \lambda \in \ZZ_{\geq 0},\  \zeta \in 2\ZZ +1
    \end{cases} \right).
\end{align}
Taking the inverse Fourier transform of these basis functions, we obtain the Fourier basis of functions lifted from the Klein bottle to the Torus
\begin{align}
    \ker \cL = \mathrm{span}\left( \begin{cases}
        \cos(2 \pi \lambda x) e^{\pi i \zeta y} & \lambda \in \ZZ_{\geq 0},\ \zeta \in 2\ZZ \\ 
         \sin(2 \pi \lambda x) e^{\pi i \zeta y}   & \lambda \in \ZZ_{\geq 0},\  \zeta \in 2\ZZ +1
    \end{cases} \right).
\end{align}
We can verify that these basis functions are invariant under the action $(\lambda, \zeta) \cdot (x,y) = ((-1)^\zeta x + \lambda, y + \zeta)$.
\end{example}
\begin{example}[Transpose and Flip]
    We take $k_1 = k_2 = 2$, and set $\varphi: \ZZ^2 \to \Aut(\ZZ^2)$ to be the following homomorphism expressed in matrix form as
\begin{equation}
    \varphi(n_1, n_2) = (-1)^{n_1} \begin{pmatrix}
        0 & 1 \\
        1 & 0 
    \end{pmatrix}^{n_2}. 
\end{equation}
Let $(\lambda_1, \lambda_2, \zeta_1, \zeta_2) \in \ZZ^{k_1} \times \ZZ^{k_2}$ denote the Fourier variables. Let us consider Fourier coefficients with $(\zeta_1, \zeta_2) = (1,1)$. We also restrict to the subspace of coefficients where $(\lambda_1, \lambda_2) \in \{-1,0,1 \}^2$. Setting $v(\lambda) = c(\lambda, (1,1))$ as a shorthand, the dual operator $\cL^\star$ acts on the the vector of coefficients by the following matrix:
\begin{align}
   \cL^\ast \rvert_{|\lambda_i| \leq 1,\ \zeta = (1,1)} v =  \frac{1}{4}\begin{pmatrix}
        4 &   &   &   &   &    &    &    &    \\
  & 4 &   &   &   &    &    &    &    \\
  &   & 4 &   &   &    &    &    &    \\
  &   &   & 2 & 2 &    &    &    &    \\
  &   &   & 2 & 2 &    &    &    &    \\
  &   &   &   &   & 3  & 1  & -1 & 1  \\
  &   &   &   &   & 1  & 3  & 1  & -1 \\
  &   &   &   &   & -1 & 1  & 3  & 1  \\
  &   &   &   &   & 1  & -1 & 1  & 3 
    \end{pmatrix} \begin{pmatrix}
        v(0,0) \\
        v(-1,-1) \\ v(1,1) 
        \\ v(-1,1) \\ v(1,-1)\\
        v(-1,0) \\ v(0,-1) \\ v(0,1)\\ v(1,0)
    \end{pmatrix}.
\end{align}
The blocks correspond to $\lambda$'s in $\ZZ^2$ on the same orbit of the the action of automorphisms on $(\lambda_1, \lambda_2) \in \ZZ^2$:
\begin{equation} \label{eq:orbits_flip_action}
    \beta \cdot \lambda = {\varphi}(\beta_1, \beta_2)^T \lambda = (-1)^{\beta_1} \begin{pmatrix}
        0 & 1 \\
        1 & 0 
    \end{pmatrix}^{\beta_2} \begin{pmatrix}
        \lambda_1 \\
        \lambda_2
    \end{pmatrix}.
\end{equation}
The orbits are illustrated in \Cref{fig:orbits_flip}.
\begin{figure}[h]
    \centering
\includegraphics[width=0.85\linewidth]{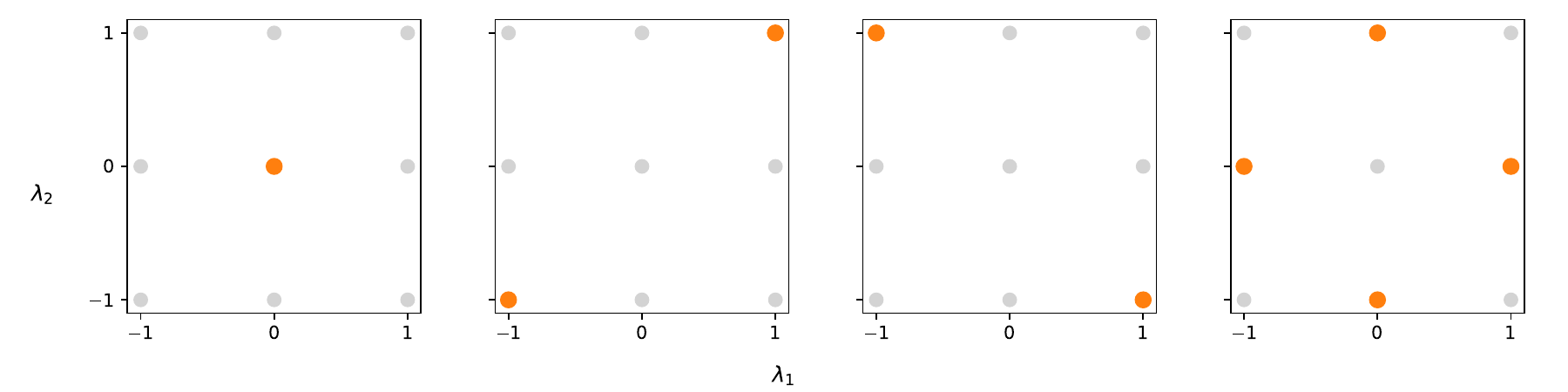}
    \caption{The orbits of the action of automorphisms of $\ZZ^2$, parametrised by \cref{eq:orbits_flip_action}, on $(\lambda_1, \lambda_2) \in \{-1,0,1 \}^2$.}
    \label{fig:orbits_flip}
\end{figure}
We can then compute the kernel of $\cL^\star$ on the subspace of coefficients, which is two-dimensional in a subspace of dimension 9, and admits an orthogonal basis
\begin{align}
  &\ker \cL^\ast \rvert_{|\lambda_i| \leq 1,\ \zeta} \\ \nonumber
  &= \{ v(-1,1) = v(1,-1),\ -v(-1,0) + v(0,-1) - v(0,1) + v(1,0) = 0 \}
\end{align}
Taking the inverse Fourier transform of the basis vectors above, we obtain a basis for the kernel of $\cL$ restricted to these Fourier modes $|\lambda_i| \leq 1, \zeta = (1,1)$:
\begin{align}
    \cF^{-1} ( \ker \cL^\ast \rvert_{|\lambda_i| \leq 1,\ \zeta})  = \mathrm{span} \{ & \sin(2\pi(x_1 - x_2)) e^{\pi i (y_1 + y_2)},\\
    &\left(\sin(2\pi x_1) -  \sin(2\pi x_2)\right) e^{\pi i (y_1 + y_2)}\} \subset \ker \cL.
\end{align}
We can then check that these are a basis for the functions that are invariant under the automorphisms $\varphi(\beta)$ restricted to the Fourier modes above. 

\end{example}

\begin{example}[Double Flip] Consider the generalised Klein bottle $\KK(2,2,B)$, where $B$ Set $\varphi: \ZZ^2 \to \Aut(\ZZ^2)$ to be the following homomorphism expressed in matrix form as
\begin{equation}
    \varphi(n_1, n_2) = \begin{pmatrix}
        (-1)^{n_1} & 0 \\
        0 & (-1)^{n_2} 
    \end{pmatrix}. 
\end{equation}
As with the example above, let us restrict to Fourier modes with $(\zeta_1, \zeta_2) = (1,1)$, and $(\lambda_1, \lambda_2) \in \{-1,0,1 \}^2$. Repeating the same procedure, we obtain a basis for the kernel of $\cL$ restricted to these Fourier modes:
\begin{align}
     \mathrm{span} \{ \sin(2\pi x_1) \sin(2\pi x_2)e^{i\pi(y_1 + y_2)} \} \subset\ker \cL.
\end{align}
We can then check that these are a basis for the functions that are invariant under the automorphisms $\varphi(\beta)$ restricted to the Fourier modes above. 
\end{example}

\section{Fourier modes of vector fields on Klein bottles} \label{appB}
We now apply the same treatment to vector fields.  We first consider local coordinate transformations along orbits of the group action on $\TT$. Note that the tangent bundle of the torus is simply a trivial bundle $\mathcal{T} \TT \cong \TT \times \RR^n$, and vectors can be parallel translated around $\TT$ as in Euclidean space,  since we are working with the flat torus. This allows us to fix a global Euclidean coordinates frame across all of the tangent spaces $T_p \TT$ in the tangent bundle $\mathcal{T}\TT$, and write vector fields as a function $V: \TT \to \RR^n$.  

We now consider how our $\ZZ_2^{k_2}$-group action on $\TT$ induces a group action on $\mathcal{T} \TT$. Recall diffeomorphisms between smooth manifolds induce diffeomorphisms between their respective tangent bundles. We recall the $\ZZ_2^{k_2}$-action is a subgroup of the group of diffeomorphism of $\TT$. The action $z \mapsto \beta \cdot z$ induces a transformation that sends a vector $v \in T_z \TT$ to a vector $\beta \cdot v \in T_{\beta\cdot z} \TT$. Since vectors transform via matrix multiplication by the Jacobian, and  $\varphi(\beta) \oplus 1_{k_2}$ is the Jacobian of the transformation, we have 
\begin{align}
\beta \cdot v = \underbrace{\begin{pmatrix}
        \varphi(\beta) & 0\\
        0 & 1_{k_2}
    \end{pmatrix}}_{A(\beta)}v
\end{align}
Note that $A(\beta) \in \GLin(n, \ZZ)$. A vector field $V$ on the torus is invariant with respect to our group action, if for all $\beta \in \ZZ_2^{k_2}$,
\begin{equation} \label{eq:vector_symmetry}
    V(\beta \cdot z) = A(\beta) V(z).
\end{equation}

We can express any vector field on $\KK$ as a vector field on $\TT$  satisfying \cref{eq:vector_symmetry}, and any vector field on $\TT$ satisfying \cref{eq:vector_symmetry} represents a vector field on $\KK$. Recall $q : \TT \to \KK$ is a local diffeomorphism. For $z \in \TT$ and $q(z) \in \KK$, we can choose a sufficiently small neighbourhood $U$ of $z$ such that $q^{-1}(q(U)) = \sqcup_{\beta} \beta \cdot U$. Because $q$ is a smooth covering, the restriction $q\rvert_U: U \to q(U)$ is a diffeomorphism, and induced isomorphisms of vector fields defined on $U$ and $q(U)$. Because the group action sends $U$ diffeomorphically to all other disjoint copies in $q^{-1}(q(U))$, any vector field defined on $q(U)$ pullsback to a vector field on $q^{-1}(q(U))$, satisfying \cref{eq:vector_symmetry}; conversely, any vector field on $U$ satisfying \cref{eq:vector_symmetry} pushes forward to a well-defined vector field on $q(U)$. Because this holds on every element of a sufficiently fine open cover of $\KK$, we deduce that any vector field $W$ on $\KK$ pulls back to a vector field $q^\ast W$ on $\TT$ satisfying \cref{eq:vector_symmetry}, and any vector field $V$ on $\TT$ satisfying \cref{eq:vector_symmetry} pushes forward to a vector field $q_\ast V$ on $\KK$; also note that if $W$ can be pullback, then $q_\ast q^\ast W = W$; and similarly if $V$ can be pushed forward, then $ q^\ast q_\ast V = V$.

Applying \Cref{lem:symmetry}, a vector field $V: \TT \to \RR^n$ satisfies \cref{eq:vector_symmetry} iff it is in the kernel of the operator
\begin{equation} \label{eq:vec_torus_klein_operator}
     \left(\cL V\right)(z) = V(z) - \frac{1}{2^{k_2}} \sum_{\beta \in \ZZ_2^{k_2}} A(\beta)V(\beta \cdot z).
\end{equation}
Writing $V = (X,Y)$ where $X: \TT \to \RR^{k_1}$ and $Y: \TT \to \RR^{k_2}$, the condition $V \in \ker \cL$ can be rewritten as thus, since $A(\beta) = \varphi(\beta) \oplus 1_{k_2}$ is block diagonal:
\begin{align}
    X(z) &= \frac{1}{2^{k_2}}  \sum_{\beta \in \ZZ_2^{k_2}} \varphi(\beta) X(\beta \cdot z) \\
    Y(z)&= \frac{1}{2^{k_2}}  \sum_{\beta \in \ZZ_2^{k_2}} Y(\beta \cdot z).
\end{align}
Note that the condition on each $Y$ coordinate is simply the scalar field condition. We focus on the $X$ coordinates of the vector field. If $X$ admits a Fourier transform, then \Cref{lem:symmetry_fourier} implies the Fourier coefficients $c= \cF(X): \ZZ^{k_1} \times \ZZ^{k_2} \to \CC^{k_1}$ are given by the kernel of the following operator 
\begin{align} \label{eq:vec_torus_klein_operator_fourier}
     (\cL_X^\star c)(\lambda, \zeta) = c( \lambda, \zeta) -\frac{1}{2^{k_2}}  \sum_{\beta \in \ZZ_2^{k_2}} (-1)^{\zeta.\beta} \varphi(\beta) c(\varphi(\beta)^T \lambda, \zeta) .
\end{align}
\begin{example}[The Standard Klein Bottle]\label{ex:standard_klein} For the Klein bottle $\KK(1,1)$, the operator $\cL_X^\star$ (\cref{eq:vec_torus_klein_operator_fourier}) on the Fourier coefficient  $c: \ZZ  \to \CC$ of the $X$ component of the vector field $V$  is 
\begin{align}
    (\cL_X^\star c)(\lambda, \zeta) = \frac{1}{2}\left(c(\lambda, \zeta) + c(-\lambda, \zeta) \right).
\end{align}
By solving for the kernel of $\cL_X^\star$ the $X$ component admits a Fourier basis 
\begin{align}
    X \in  \mathrm{span} \bigg( &\{\sin(2 \pi \lambda x) e^{\pi i \zeta y} \ \mid \ \lambda \in \ZZ_{\geq 0},\ \zeta \in 2\ZZ \} \\ \nonumber
                    &\cup\ \{\cos(2 \pi \lambda x) e^{\pi i \zeta y}  \ \mid \ \lambda \in \ZZ_{\geq 0},\ \zeta \in 2\ZZ + 1\} \bigg) .
\end{align}
Because the $Y$ coordinate satisfies the same constraints as the scalar field case, 
\begin{align}
    Y \in  \mathrm{span} \bigg( &\{\cos(2 \pi \lambda x) e^{\pi i \zeta y} \ \mid \ \lambda \in \ZZ_{\geq 0},\ \zeta \in 2\ZZ \} \\ \nonumber
                    &\cup\ \{\sin(2 \pi \lambda x) e^{\pi i \zeta y}  \ \mid \ \lambda \in \ZZ_{\geq 0},\ \zeta \in 2\ZZ + 1\} \bigg) .
\end{align}
Thus, vector fields on $\TT$ with the following Fourier series pushes forward to a vector field on the Klein bottle $\KK$:
\begin{align}
    V(e^{2\pi i x}, e^{\pi i y}) =\sum_{\lambda  \geq 0}   \sum_{\zeta \in 2\ZZ} 
    \begin{pmatrix}
        r_{\lambda, \zeta} \sin(2\pi \lambda x) \\
        s_{\lambda, \zeta}\cos(2\pi \lambda x)
    \end{pmatrix}  e^{i \pi \zeta y}  + \sum_{\zeta \in 2\ZZ + 1} 
    \begin{pmatrix}
        r_{\lambda, \zeta} \cos(2\pi \lambda x) \\
        s_{\lambda, \zeta}\sin(2\pi \lambda x)
    \end{pmatrix}  e^{i \pi \zeta y} 
\end{align}
\end{example}

\section{Technical lemmas for group actions and symmetries}
Consider the group of affine transformations $\Aff(\RR^n)$ on $\RR^n$. For $b \in \RR^n$, and $A \in \GLin(n, \RR)$ an invertible matrix, we have an affine action on $\RR^n$ given by
\begin{equation}
  (b,A) \cdot  x = A.x + b
\end{equation}
The affine group is a semi-direct product $\Aff(\RR^n) = \RR^n \rtimes \GLin(n, \RR)$, where 
\begin{equation}
    (b', A') \ast (b, A) = (b' + A'b, A'A).
\end{equation}
We can check that this is the choice of product structure such that it is compatible with the composition of affine transformations, satisfying 
\begin{equation}
    (b', A') \cdot ((b, A) \cdot x) = ((b', A') \ast (b, A)) \cdot x.
\end{equation}
Suppose we restrict to a subgroup of affine transformations, parametrised over another group $G$. In other words, we have a group homomorphism $\Phi: G \to \Aff(\RR^n)$, which we write as 
\begin{equation}
    \Phi:\ g \mapsto (b(g), A(g))
\end{equation}
where $A: G \to \GLin(n, \RR)$ and $b: G \to \RR^n$ are functions. The group homomorphism condition on $\Phi$ enforces that $A$ is a group homomorphism. As for $b$, since
\begin{equation}
    b(g \cdot h) = A(g)b(h)  + b(g).
\end{equation}
Unless $A$ maps every element of $G$ to the identity matrix, $b$ is notably not a homomorphism .  

\begin{remark}
    In the case of the generalised Klein bottle, we consider $G = \ZZ^{k_1} \rtimes_\varphi \ZZ^{k_2}$ to be the semi-direct product between $\ZZ^{k_1}$ and $\ZZ^{k_2}$, specified by $\varphi: \ZZ^{k_2} \to \Aut (\ZZ^{k_1})$. Writing an element of $G$ as $(\lambda, \zeta)$ where $\lambda \in \ZZ^{k_1}$, and $\zeta \in \ZZ^{k_2}$, we set $b = (b_1, b_2)$ where $b_1(\lambda, \zeta) = \lambda$ and  $b_2(\lambda, \zeta) = \zeta$ simply embeds the integer tuple as real coordinates. Thus
\begin{align}
    b_1((\lambda', \zeta') \cdot (\lambda, \zeta)) &= A_{11}(\lambda' , \zeta')b_1(\lambda, \zeta) +  A_{12}(\lambda' , \zeta')b_2(\lambda , \zeta) + b_1(\lambda' , \zeta') \\
   \implies \varphi(\zeta') \lambda &=A_{11}(\lambda' , \zeta')\lambda +  A_{12}(\lambda' , \zeta')\zeta \\
   b_2((\lambda', \zeta') \cdot (\lambda, \zeta)) &= A_{21}(\lambda' , \zeta')b_1(\lambda, \zeta) +  A_{22}(\lambda' , \zeta')b_2(\lambda , \zeta) + b_2(\lambda' , \zeta') \\
   \implies \zeta &= A_{21}(\lambda' , \zeta')\lambda +  A_{22}(\lambda' , \zeta')\zeta.
\end{align}
Since these relations hold for arbitrary $(\lambda, \zeta)$, we conclude that $A_{11} = \varphi$, $A_{12} = A_{21} = 0$. Thus 
\begin{equation}
    (\lambda, \zeta) \cdot (x,y) = (\varphi(\zeta).x + \lambda , y + \zeta).
\end{equation}
\end{remark}

\begin{lemma} \label{lem:symmetry_fourier}
Let $\TT$ be an $n$-torus equipped with a $G$-action, where $G$ is a finite group, and acts on $z = (e^{2\pi i \theta_1}, \ldots, e^{2\pi i \theta_n}) \in \TT$ by the transformation
    \begin{equation}
        g \cdot z = (e^{2\pi i (A(g)\theta +b(g))_1}, \ldots, e^{2\pi i (A(g)\theta +b(g))_n})
    \end{equation}
where $A: G \to \GLin(n, \ZZ)$ is a homomorphism, and $b(g)$ satisfies 
\begin{equation*}
    b(g \cdot h) =  A(g)b(h) + b(g) \mod 1
\end{equation*}
so that the above transformation is a group action. Consider the dual operator $\cL^\star : \CC^n[\ZZ^n] \to \CC^n[\ZZ^n]$ of $\cL: \CC^n[\TT] \to \CC^n[\TT]$ in \cref{eq:general_symmetry_operator}, defined via
\begin{equation}
    \cL^\star = \cF\cL \cF^{-1}.
\end{equation}
Then $\cL^\star$ is explicitly given on $c \in \CC[\ZZ^n]$, $k \in \ZZ^n$ by
\begin{equation} \label{eq:torus_sym_fourier_operator}
    (\cL^\star c)(k) = c(k) - \frac{1}{|G|} \sum_{g \in G} \chi(g^{-1}) c(A(g^{-1})^T k) e^{2\pi i k^T A(g^{-1})b(g)}.
\end{equation}
\end{lemma}

\begin{proof}
    Consider $c \in \CC^n[\ZZ^n]$. 
    \begin{align*}
        \cL^\star c &= \cF \cL \cF^{-1} c = \cF \left(\cL \int  c(k)e^{2\pi i k } dk \right) = \cF \left( \int  \cL \left(c(k)e^{2\pi i k } \right)dk \right)
    \end{align*}
     Where we have the last equality due to linearity. Thus 
     \begin{align*}
         \cL \left(c(k)e^{2\pi i k } \right) &= c(k)e^{2\pi i k }  - \frac{1}{|G|} \sum_{g \in G} \chi(g^{-1}) c(k) e^{2\pi i k } \\
         \implies  \int  \cL \left(c(k)e^{2\pi i k } \right)dk &= \cF^{-1}(c) - \frac{1}{|G|} \sum_{g \in G}\chi(g^{-1})  \int c(k) e^{2\pi i k^T ( A(g) x + b(g))} dk  \\
         &= \cF^{-1}(c) - \frac{1}{|G|} \sum_{g \in G}\chi(g^{-1})  \int c(A(g^{-1})^T q) e^{2\pi i q^T A(g^{-1})b(g)} e^{2\pi i q^Tx} dq \\
         &= \cF^{-1}(c) - \frac{1}{|G|} \int \sum_{g \in G}\chi(g^{-1})   c(A(g^{-1})^T q) e^{2\pi i q^T A(g^{-1})b(g)} e^{2\pi i q^Tx} dq \\
         &= \cF^{-1} \left( c(q) - \frac{1}{|G|} \sum_{g \in G}\chi(g^{-1})   c(A(g^{-1})^T q) e^{2\pi i q^T A(g^{-1})b(g)}\right)
     \end{align*}
 Applying the $\cF$ on both sides, 
 \begin{align*}
      (\cL^\star c)(q) = \cF \left( \int  \cL \left(c(k)e^{2\pi i k } \right)dk \right)(q) &= c(q) - \frac{1}{|G|} \sum_{g \in G}\chi(g^{-1})   c(A(g^{-1})^T q) e^{2\pi i q^T A(g^{-1})b(g)}.
 \end{align*}
\end{proof}
\begin{remark} \label{lem:block}
    The operator $\cL^\star$ in \cref{eq:torus_sym_fourier_operator} is block diagonal: partitioning $\ZZ^n$ into the space of orbits $\ZZ^n/G$ under the induced group action on the dual Fourier space 
    \begin{equation}
        g \cdot k  = A(g^{-1})^T k,
    \end{equation}
    and writing $\CC^n[\ZZ^n]$ as a direct sum over the finite dimensional subspaces $\CC^n[\ZZ^n] = \bigoplus_{[k] \in \ZZ^n/G} \CC^n[[k]]$, the operator $\cL^\star$ can be written as a direct sum of finite dimensional matrices on each subspace $\CC^n[[k]]$.
\end{remark}

\begin{lemma} \label{lem:symmetry}
    Let $X$ be a set equipped with a $G$-action, where $G$ is a finite group. Consider also $\chi: G \to \GLin(n, \CC)$ a group homomorphism, such that $\chi(g \ast h) = \chi(g) \chi(h)$. Then a function $f: X \to \CC^n$ satisfies $f(g\cdot x) = \chi(g)f(x)$ for all $g \in G$, iff it is in the kernel of the following linear operator $\cL$ on $\CC^n$-valued functions on $X$:
    \begin{equation} \label{eq:general_symmetry_operator}
       \left(\cL f\right)(x) = f(x) - \frac{1}{|G|} \sum_{g \in G} \chi(g^{-1})f(g\cdot x).
    \end{equation}
\end{lemma}
\begin{proof}
    Suppose $f(g\cdot x) = \chi(g)f(x)$. Then 
    \begin{align*}
        \sum_{g \in G} \chi(g^{-1})f(g\cdot x) = \sum_{g \in G} \chi(g^{-1})\chi(g)f(x) = \sum_{g \in G} \chi(g^{-1} g)f(x) = \sum_{g \in G} f(x) = |G| f(x).
    \end{align*}
    Thus $\left(\cL f\right)(x) = f(x) - \frac{1}{|G|} \sum_{g \in G} \chi(g^{-1})f(g\cdot x) = 0$, i.e. $f \in \ker \cL$.

    Conversely, if $f \in \ker \cL$, then for all $h \in G$,
    \begin{align*}
        f(h \cdot x)  &= \frac{1}{|G|}\sum_{g \in G} \chi(g^{-1})f(g\cdot h \cdot x) \\
                      &= \frac{1}{|G|}\sum_{g' \in G} \chi((g' h^{-1})^{-1})f(g' \cdot x) \\
                      &= \frac{1}{|G|}\sum_{g' \in G} \chi(h)\chi({g'}^{-1})f(g' \cdot x)\\
                       &= \chi(h) \left(\frac{1}{|G|}\sum_{g' \in G} \chi({g'}^{-1})f(g' \cdot x) \right) = \chi(h) f(x).
    \end{align*}
\end{proof}

\begin{remark} \label{rmk:laplacian}
    If $\chi(g) = 1$, then the operator $\cL$ in \cref{eq:general_symmetry_operator} can be interpreted as a \emph{graph Laplacian}. For each orbit  $x_1, \ldots, x_n$ , we construct a multigraph where the vertex set consists of the  $x_1, \ldots, x_n$;  For each group element $g \in G$, we attach an edge between every $x_i$ and $x_j$ such that $x_j = g \cdot x_i$. Edges can thus be unambigously indexed by $(x_i,g)$. While this graph is presented here as a purely combinatorial construction, we note that this graph can be viewed as a categorical object called the \emph{action groupoid} \cite{Brown1987FromSurvey}.  The \emph{coboundary operator} $d$ of the multigraph (interpreted as a cellular complex) sends vertex functions to an edge-valued function. Since the graph is complete, for any edge $(x_i,  g)$, the image of the coboundary operator evaluated at that edge is $df(x_i,  g) = f(x_j) - f(x_i) = f(g \cdot x_i) - f(x_j)$. Vertex functions in the kernel of the coboundary operator corresponds to functions that satisfy $f(g \cdot x) = f(x)$. Since the Laplacian is defined as $L = d^\dagger d$, we leave it as an exercise for the reader to verify that the Laplacian is proportional to $\cL$ in \cref{eq:general_symmetry_operator} up to a constant. 
    
    In the more general case where $\chi(g)$ is unitary for all $g$, the operator $\cL$ can be interpreted as a \emph{cellular sheaf graph Laplacian} \cite{Hansen2019TowardSheaves}, following a similar construction. 
\end{remark}

\end{document}